\documentclass[reqno,10pt]{amsart}
\usepackage{CJK}
\usepackage{hyperref}
\usepackage{cite}
\usepackage{amsmath}
\usepackage{mathrsfs}

\makeatletter
\@namedef{subjclassname@2020}{%
	\textup{2020} Mathematics Subject Classification}
\makeatother

\numberwithin{equation}{section}

\newtheorem{theorem}{Theorem}[section]
\newtheorem{lemma}[theorem]{Lemma}
\newtheorem{definition}[theorem]{Definition}

\newtheorem{proposition}[theorem]{Proposition}
\newtheorem{remark}[theorem]{Remark}

\allowdisplaybreaks

\begin{document}
	
	\title[\hfil On the solutions of nonlocal 1-Laplacian equation with $L^1$-data] {On the solutions of nonlocal 1-Laplacian equation with $L^1$-data}
	
		\author[D. Li and C. Zhang  \hfil \hfilneg]
	{Dingding Li  and Chao Zhang$^*$}
	
	\thanks{$^*$ Corresponding author.}
	
	\address{Dingding Li \hfill\break School of Mathematics, Harbin Institute of Technology, Harbin 150001, China}
	\email{}

	\address{Chao Zhang  \hfill\break School of Mathematics and Institute for Advanced Study in Mathematics, Harbin Institute of Technology, Harbin 150001, China}
	\email{}

	
	\maketitle
	
\begin{abstract}
	We study the solutions to a nonlocal 1-Laplacian equation given by 
		\begin{align*}
	2\text{P.V.}\int_{\mathbb{R}^N}\frac{u(x)-u(y)}{|u(x)-u(y)|} \frac{dy}{|x-y|^{N+s}}=f(x) \quad \textmd{for } x\in \Omega,
	\end{align*}
with Dirichlet boundary condition $u(x)=0$ in $\mathbb R^N\backslash \Omega$ and nonnegative $L^1$-data. By investigating the asymptotic behaviour of renormalized solutions $u_p$ to the nonlocal $p$-Laplacian equations as $p$ goes to $1^+$, we introduce a suitable definition of solutions and prove that  the limit function $u$ of $\{u_p\}$ is a solution of the nonlocal $1$-Laplacian equation above. 
	\end{abstract}

\section{Introduction}
\thispagestyle{empty}
\label{sec1}

Let $\Omega\subset \mathbb R^N$ be a smooth bounded domain. In this paper, we consider the following nonlocal 1-Laplacian Dirichlet problem:
\begin{align}
		\label{1.1}
		\begin{cases}
			\mathcal{L}_1u=f &\text{in }\Omega,\\
			u=0 &\text{in }\mathbb{R}^N\backslash\Omega,
		\end{cases}
\end{align}
where $0\le f\in L^1(\Omega)$ and the operator $\mathcal{L}_p$ is given by
\begin{align*}
		\mathcal{L}_pu(x):=&(-\Delta)^s_pu(x)\\
		=&2\text{P.V.}\int_{\mathbb{R}^N}\frac{|u(x)-u(y)|^{p-2}(u(x)-u(y))}{|x-y|^{N+sp}}dy, \quad x\in \Omega,
\end{align*}
with $p\in [1,+\infty)$, $s\in (0,1)$ and \text{P.V.} being a commonly used abbreviation in the principal value sense.
	
The operator $\mathcal{L}_p$, being known as the celebrated fractional (nonlocal) $p$-Laplacian, emphasizes the L\'{e}vy process which indicates the emergence of the jump diffusion. In recent years, with the development of  fractional Sobolev space,  there is an extensive literature related to nonlocal $p$-Laplacian equations.  Existence, uniqueness, regularity and the properties of solutions are treated under different situations; see \cite{AAB10,KPU11,LPPS15,AAB19,KMS15} and references therein.
	
Problem \eqref{1.1} can be seen as a nonlocal counterpart of the following $1$-Laplacian Dirichlet problem:
\begin{align}
	\label{1.4}
	\begin{cases}
	\displaystyle
		-\Delta_1u:=-\text{div}\left( \frac{\nabla u}{|\nabla u|}\right) =f&\text{in }\Omega,\\
		u=0&\text{on }\partial\Omega.
	\end{cases}
\end{align}
This equation has been usually addressed as the limit problem for $p$-Laplacian type equations
when $p$ goes to $1^+$. When the right-hand side term $f$ belongs to different settings, it was studied by Kawohl in \cite{Ka90} for constant data, 
by Cicalese and Trombetti in \cite{CT03} for data belonging to the Lorentz space $L^{N,\infty}(\Omega)$ and by Mercaldo, Segura de Le\'{o}n and Trombetti in \cite{MST08} for general data belonging to $W^{-1,\infty}(\Omega)$ (see also \cite{De99} for the case with critical Sobolev exponent). In the framework of $L^1$ setting, since no almost everywhere finite solution can be expected, Mercaldo, Segura de Le\'{o}n S and Trombetti \cite{MST09} investigated the behaviour, as $p$ goes to $1^+,$ of the renormalized solutions to the $p$-Laplacian problems and proved the existence of renormalized solutions.

The notion of renormalized solution was first introduced by DiPerna and Lions in the study of the Boltzmann equation in \cite{DL89}. There are plenty of results on renormalized solutions of $p$-Laplacian type equations based on the existence of weak solutions. For the nonlinear elliptic problem:
\begin{align}
	\label{1.3}
	\begin{cases}
		-\text{div}(a(x,\nabla u))=f&\text{in }\Omega,\\
		u=0&\text{on }\partial\Omega,
	\end{cases}
\end{align}
under the assumption that $p>2-\frac{1}{N}$, even in the case of the data is a measure, Boccardo and Gallou\"{e}t gave the existence of solutions in the sense of distribution and proved that  the renormalized solution of \eqref{1.3} is a distributed solution in \cite{BG89,BG92}. Maso, Murat, Orsina and Prignet showed in \cite{DMOP99} that, the restriction on $p$ could be eliminated, that is the renormalized solution of \eqref{1.3} is a distributed solution for all $p>1$. At the same time, the study of renormalized solutions was also adapted to general nonlinear elliptic or parabolic problems \cite{BM97, BM01, BR98,BGD93,DMOP99,ZS10}.  

As far as the parabolic $1$-Laplacian operator are regarded, the pioneering works can be found in a series of papers by Andreu, Ballesteler, Caselles and Maz\'{o}n \cite{ABCM00, ABCM01, ABCM012, ACM04}. Indeed, in \cite{ABCM01}, the authors characterized the imprecise quotient $\frac{\nabla u}{|\nabla u|}$, through the Anzellotti-Frid-Chen's Pairing Theory (see \cite{An83, CF99}).  This theory allows them to introduce a vector field $z\in L^\infty(\Omega, \mathbb R^N)$ which plays the role of $\frac{\nabla u}{|\nabla u|}$. We also remark that in \cite{LS22}, the authors investigated the  inhomogeneous $1$-Laplace evolution equation in another way.

Finally, we would like to turn to the nonlocal equations. To the best of our knowledge, it seems that there are few results on this issue. In particular, Andreu, Maz\'{o}n, Rossi and Toledo \cite{AMRT09, AMRT08} considered the following nonlocal diffusion equation with different boundary conditions
	$$
	u_t(x,t)=\int_{\mathbb R^n} J(x-y)|u(t,y)-u(t,x)|^{p-2}(u(t,y)-u(t,x))\,dy,
	$$
	where $J:\mathbb R^N\to \mathbb R$ is a nonnegative, radial, continuous function, strictly positive in $B(0,1)$, vanishing in $\mathbb R^N \backslash B(0,1)$ and such that $\int_{\mathbb R^N} J(z)\,dz=1$. The authors  studied $1<p<+\infty$ as well as the extreme cases $p=1$
 and limits $p\to +\infty$, see also \cite{MRT19}. Novaga and Onoue \cite{MF23}  obtained the regularity of solutions to a nonlocal variational problem related to the image denoising model:
	$$
	\min\{\mathcal F_{K,f}(u): u \in BV_K(\mathbb R^n)\cap L^2(\mathbb R^n)\},
	$$
	where the functional $\mathcal F_{K,f}$ is defined as
	$$
	\mathcal F_{K,f}(u):=\frac 12\int_{\mathbb R^n}\int_{\mathbb R^n}  K(x-y)|u(x)-u(y)|
\,dxdy+\frac 12 \int_{\mathbb R^n}(u(x)-f(x))^2\,dx,
	$$
	the function $K$ is a kernel singular at the origin, and a typical example is the function $|x|^{-(N+s)}$.
The authors showed that, in two dimensions, minimizers have the same H\"{o}lder regularity as the original image. We also refer to \cite{BR18,KRB17,ZZ20} for the nonlocal equations with variable exponents.

Concerning the nonlocal $1$-Laplacian equation \eqref{1.1}, the existence of $(s,1)$-harmonic function and the equivalence between those functions and minimizers were investigated in \cite{BU21}. Moreover, under the assumption that $f$ belongs to a certain closed ball in $L^\frac{N}{s}(\Omega)$, Bucur  in \cite{BU23} studied the minimizer of $\mathcal{F}^{s_p}_p$ (which is equivalent to the weak solution of nonlocal $p$-Laplacian)
\begin{align*}
	\mathcal{F}^{s_p}_p(u):=\frac{1}{2p}\int\int_{\mathbb{R}^{2N}}\frac{|u(x)-u(y)|^p}{|x-y|^{N+s_pp}}dxdy-\int_{\Omega}fu\,dx,
\end{align*}
where $p$ is small enough and $s_p=N+s-\frac{N}{p}\in (s,1)$, and obtained the existence, asymptotics and flatness results.

Motivated by the paper \cite{MST09}, our main aim is to deal with the solutions to the nonlocal 1-Laplacian equation with homogeneous Dirichlet boundary condition \eqref{1.1} and nonnegative $L^1$-data. To be more precise, we concentrate on the behaviour of the renormalized solutions $u_p$ when $p$ goes to $1^+$, and prove that, up to a subsequence, 
\begin{align*}
&u_p\rightarrow u \quad \text{a.e. in }\Omega,
\end{align*}
\begin{align*}
&\frac{|u_p(x)-u_p(y)|^{p-2}\left( u_p(x)-u_p(y)\right) }{|x-y|^{\left( \frac{N}{p}+s\right)(p-1) }}\rightharpoonup Z(x,y) \quad \text{in } L^q(A_\Omega)
\end{align*}
and
\begin{align*}
|u_p(\cdot)|^{p-2}u_p(\cdot)\rightharpoonup Z(\cdot,y) \quad\text{in }L^q(\Omega)
\end{align*}
for any fixed $y\in \mathbb{R}^N\backslash\Omega_0$,  where $1\le q<\frac{N}{N-s}$, $A_\Omega:=(\Omega\times\Omega_0)\cup(\Omega_0\times\Omega)$ and  
$\Omega_0:=\Omega+B_1(0)$ with $B_1(0)$ donating the unit ball in $\mathbb{R}^N$ centered on the origin. Different from the local case, the normal vector field $\frac{\nabla u}{|\nabla u|}$ in classical framework is taken by the fractional ratio $\frac{u(x)-u(y)}{|u(x)-u(y)|}$, which leads to the singular set is more difficult to interpret and describe.  After finding a suitable function $Z$ with  $Z\in \text{sgn}(u(x)-u(y))$, where sgn is the generalized sign function satisfying $\text{sgn}(0)=[-1,1]$, we show that this function $u$ is a solution to problem \eqref{1.1}. 
	
\medskip

To state our result, we introduce now  the notion of solution to such a problem in the sense of renormalization.
	
\begin{definition}
	\label{def1.1}
	A measurable function $u:\mathbb{R}^N\rightarrow\bar{\mathbb{R}}$ is said to be a solution to problem \eqref{1.1} if the following conditions hold:
	\begin{itemize}
	\item [(1)] $T_k(u)\in W^{s,1}_0(\Omega)$ for any fixed $k>0;$
	\item [(2)] There exists a function $Z$ such that $Z_k\in L^\infty(\mathbb{R}^N\times\mathbb{R}^N)$ and $\|Z_k\|_{L^\infty(\mathbb{R}^N\times\mathbb{R}^N)}\le 1,$ where $Z_k$ is defined as
	\begin{align*}
		Z_k:=Z\mathcal{X}_{\left\lbrace (x,y)\in \mathbb{R}^N\times\mathbb{R}^N:u(x)<k\text{ and }u(y)<k\right\rbrace};
	\end{align*}
	\item [(3)]For fixed $k>0$, there holds
	\begin{align*}
	Z\left( u(x)-u(y)\right)=\left| u(x)-u(y)\right|  
	\end{align*}
		in the domain
	\begin{align*}
	\left\lbrace (x,y)\in\mathbb{R}^N\times\mathbb{R}^N:u(x)<k\text{ and }u(y)<k\right\rbrace;
	\end{align*}
	\item [(4)] $Z$ is an anti-symmetry function, that is
	\begin{align*}
	Z(x,y)=-Z(y,x)
	\end{align*}
	for all $(x,y)\in \mathbb{R}^N\times\mathbb{R}^N;$
	\item [(5)] For all $\varphi\in W^{1,\infty}_0(\Omega)$ and $S\in W^{1,\infty}(\mathbb{R})$ with compact support, satisfying $S(t)$ is monotone in $(0,+\infty)$, $S'(t)\varphi(x)\ge 0$ for all $t\ge0$ and $x\in\Omega$, there holds
	\begin{align}
	\label{4.2}
	\int_{\mathcal{C}_\Omega}Z_{\infty}\frac{\left( S(u)\varphi\right)(x)-\left(S(u)\varphi\right)(y) }{|x-y|^{N+s}}dxdy\le \int_{\Omega}fS(u)\varphi dx,
	\end{align}
	where $Z_\infty=Z\mathcal{X}_{\left\lbrace (x,y)\in \mathbb{R}^N\times\mathbb{R}^N:u(x)<+\infty\ and\ u(y)<+\infty\right\rbrace }$ and
	\begin{align*}
	\mathcal{C}_\Omega:=(\mathbb{R}^N\times\mathbb{R}^N)\backslash\left[ (\mathbb{R}^N\backslash\Omega)\times(\mathbb{R}^N\backslash\Omega)\right].
	\end{align*}
	\end{itemize}
\end{definition}

Next we introduce our main result as follows:

\begin{theorem}
		\label{the1.2}
		There exists a measurable function $u$ such that it is a solution to problem \eqref{1.1} in the sense of Definition \ref{def1.1}.
\end{theorem}

\begin{remark}
We remark that we have to impose  the monotonicity and nonnegativity conditions in (5) of Definition \ref{def1.1} due the technical reason from the Fatou's lemma. Another approximation method has been proposed to study the nonlocal $1$-Laplacian in \cite{BU23}, which may be useful to solve this problem. However, there are still some difficulties in dealing with the $L^1$-data problem.
\end{remark}

The rest of this paper is organized as follows. In Section \ref{sec2}, we state some notations and useful tools regarding to the fractional Sobolev spaces and renormalized solutions. Section \ref{sec3} is devoted to studing the relationship between fractional Sobolev space with different parameters and the asymptotic behaviour of renormalized solutions $u_p$ to the nonlocal $p$-Laplacian equations as $p$ goes to $1^+$. We will give the proof of Theorem \ref{the1.2}  in Section \ref{sec4}.

\section{ Preliminaries}
\label{sec2}
	
In this section, we introduce some notation used later. We will say that $\left\lbrace u_p\right\rbrace $ is a sequence and consider subsequences of it. We denote by $|z|$ the Euclidean norm of $z\in \mathbb{R}^N$, $\text{meas}\left(  E\right)  $ donates the Lebesgue measure of a measurable subset $E$ of $\mathbb{R}^N$.
	
For $1<q<+\infty$, the Lorentz space $L^{q,\infty}(\Omega)$ is the space of Lebesgue measurable functions $u$ such that
	\begin{align*}
		\sup_{t>0}\left\lbrace 
		t \left( \text{meas}\left( \left\lbrace x\in\Omega:|u(x)|>t\right\rbrace \right) \right) ^\frac{1}{q}
		\right\rbrace <+\infty.
	\end{align*}\par 
For $0<\alpha<1$ and $1\le p<+\infty$, the fractional Sobolev space $W^{\alpha,p}(\mathbb{R}^N)$ is defined as
	\begin{align*}
		\left\lbrace 
		u\in L^p(\mathbb{R}^N):[u]_{W^{\alpha,p}(\mathbb{R}^N)}:=\left( \int_{\mathbb{R}^N\times\mathbb{R}^N}\frac{|u(x)-u(y)|^p}{|x-y|^{N+sp}}dxdy\right)^\frac{1}{p}<+\infty 
		\right\rbrace .
	\end{align*}
For a smooth bounded domain $\Omega\subset\mathbb{R}^N$, we denote by $W^{\alpha,p}_0(\Omega)$ a fractional Sobolev space, which is
	\begin{align*}
		W^{\alpha,p}_0(\Omega):=\left\lbrace 
		u\in W^{\alpha,p}(\mathbb{R}^N):u=0\ \text{a.e. in }\mathbb{R}^N\backslash\Omega
		\right\rbrace.
	\end{align*}
We refer to \cite{DPV12} for more information about the space $W^{\alpha,p}(\mathbb{R}^N)$.

From \cite{ZZ20}, we give the definition of renormalized solution of the following fractional $p$-Laplacian problem:
	\begin{align}
		\label{2.1}
		\begin{cases}
			\mathcal{L}_pu=f &\text{ in }\Omega,\\
			u=0 &\text{ in }\mathbb{R}^N\backslash\Omega,
		\end{cases}
	\end{align}
where $f$ is a nonnegative function belonging to $L^1(\Omega)$ and $p>1$.
	\begin{definition}
		\label{def2.1}
		$u_p$ is said to be a renormalized solution to problem \eqref{2.1} if $u_p$ satisfies:
		\begin{itemize}
			\item [(1)] For any fixed $k>0$, $T_k(u_p)\in W^{s,p}_0(\Omega)$, where the truncation function $T_k$ is defined by
		\begin{align*}
			T_k(t)=\max\left\lbrace -k,\min\left\lbrace k,t\right\rbrace \right\rbrace 
		\end{align*}
		for any $t\in \mathbb{R};$
		
		\item [(2)]
		\begin{align*}
			\lim\limits_{h\rightarrow+\infty}\int_{\left\lbrace (x,y)\in\mathcal{C}_\Omega:\left( u_p(x),u_p(y)\right)\in R_h \right\rbrace }
			\frac{|u(x)-u(y)|^{p-1}}{|x-y|^{N+sp}}dxdy=0,
		\end{align*}
		where
		\begin{align*}
			R_h=\left\lbrace 
			(v,w)\in \mathbb{R}^2:\max\left\lbrace |v|,|w|\right\rbrace\ge h+1\ \text{and }\min\left\lbrace |v|,|w|\right\rbrace \le h
			\right\rbrace ;
		\end{align*}\par 
		\item [(3)] For any $\varphi\in C^{\infty}_0(\Omega)$ and $S\in W^{1,\infty}(\mathbb{R})$ with compact support, there holds
		\begin{align*}
			\begin{aligned}
				&\quad\int_{\mathcal{C}_\Omega}\frac{|u_p(x)-u_p(y)|^{p-2}\left( u_p(x)-u_p(y)\right)\left[ \left( S(u_p)\varphi\right)(x)-\left( S(u_p)\varphi\right)(y) \right]  }{|x-y|^{N+sp}}dxdy\\
				&=\int_{\Omega}fS(u_p)\varphi dx.
			\end{aligned}
		\end{align*}
			\end{itemize}
	\end{definition}\par 

	\begin{remark}
		\label{rem2.2}
		In (3) of Definition \ref{def2.1}, it is enough for us to assume that $\varphi\in W^{s,p}(\Omega)\cap L^\infty(\Omega)$ and $S\in W^{1,\infty}(\mathbb{R})$ with compact support, such that	$S(u_p)\varphi\in W^{s,p}_0(\Omega)$,
		which could be proved by the method in the proof of \cite[Theorem 3.4]{ZZ20}.
	\end{remark}
	
	
Next, for the convenience of the readers, we give some lemmas that can be proved by following the arguments given in \cite{ZZ20}.

	\begin{lemma}
		\label{lem2.3}
		Let $u_p$ be a renormalized solution to problem \eqref{2.1}, then
		\begin{align*}
			T_k(u_p)\in W^{s,p}_0(\Omega)
		\end{align*}
		and for fixed $\delta,\gamma>0$, set 
		\begin{align*}
			S_{\delta,\gamma,h}(t)=T_\delta(\gamma G_h(t)),
		\end{align*} 
		where
		\begin{align*}
			G_h(t)=t-T_h(t),
		\end{align*} 
		we have
		\begin{align*}
			\lim\limits_{h\rightarrow\infty}\int_{\mathcal{C}_\Omega}
			\left[ S_{\delta,\gamma,h}(u_p)(x)-S_{\delta,\gamma,h}(u_p)(y)\right]\frac{|u_p(x)-u_p(y)|^{p-2}\left( u_p(x)-u_p(y)\right)}{|x-y|^{N+sp}}dxdy=0.
		\end{align*}
	\end{lemma}
	\begin{lemma}
		\label{lem2.4}
		The renormalized solution $u_p$ to problem \eqref{2.1} for fixed $p>1$ and $s\in(0,1)$ satisfies
		\begin{align*}
			\left[ T_k(u_p)\right]_{W^{s,p}(\mathbb{R}^N)}^p\le k\|f\|_{L^1(\Omega)}
		\end{align*}
		for fixed $k>0$.
	\end{lemma}
	\begin{proof}
		For the function $0\le f\in L^1(\Omega)$, taking a sequence $f_n=T_n(f)$ such that
		\begin{align*}
			0\le f_n\le f
		\end{align*}
		and
		\begin{align*}
			f_n\rightarrow f\quad \text{in }L^1(\Omega),
		\end{align*}
		then, by calculus of variations, for any $n\in\mathbb{N}\cup\left\lbrace 0\right\rbrace $, there exists a nonnegative weak solution $u_{n,p}$ in $W^{s,p}_0(\Omega)$ to problem \eqref{2.1}, where $f$ is replaced by $f_n$. That is, for any $v\in W^{s,p}_0(\Omega)$, there holds
		\begin{align*}
			\begin{aligned}
				\int_{\mathcal{C}_\Omega}\frac{|u_{n,p}(x)-u_{n,p}(y)|^{p-2}(u_{n,p}(x)-u_{n,p}(y))(v(x)-v(y))}{|x-y|^{N+sp}}dxdy=\int_{\Omega}f_nvdx.
			\end{aligned}
		\end{align*}
		We take $v=T_k(u_{n,p})$ as a test function to have
		\begin{align*}
			\begin{aligned}
				&\quad\int_{\mathcal{C}_\Omega}\frac{|T_k(u_{n,p})(x)-T_k(u_{n,p})(y)|^p}{|x-y|^{N+sp}}dxdy\\
				&\le\int_{\mathcal{C}_\Omega}\frac{|u_{n,p}(x)-u_{n,p}(y)|^{p-2}\left( u_{n,p}(x)-u_{n,p}(y)\right)}{|x-y|^{N+sp}}\left[ T_k(u_{n,p})(x)-T_k(u_{n,p})(y)\right]dxdy\\
				&=\int_{\Omega}f_nT_k(u_{n,p})dx \le k\|f_n\|_{L^1(\Omega)} \le k\|f\|_{L^1(\Omega)}.
			\end{aligned}
		\end{align*}
		In addition, similar to the proof of \cite[Lemmas 3.2 and 3.3]{ZZ20}, we could prove
		\begin{align*}
			T_k(u_{n,p})\rightarrow T_k(u_p)\quad \text{in }W^{s,p}_0(\Omega)\ \text{as }n\rightarrow+\infty.
		\end{align*}
		Therefore,  we arrive at
		\begin{align*}
			\left[ T_k(u_p)\right]_{W^{s,p}(\mathbb{R}^N)}^p\le k\|f\|_{L^1(\Omega)}.
		\end{align*}
	\end{proof}

\section{The asymptotic behaviour}
\label{sec3}
In this section, we study the asymptotic behaviour of the renormalized solutions sequence $\left\lbrace u_p\right\rbrace $ as $p$ goes to $1^+$. We not only find a measurable function $u$ satisfying $u_p\rightarrow u$ by applying property of Lorentz space, but also find a function $Z$, which will play the role of $\frac{u(x)-u(y)}{|u(x)-u(y)|}$ in a imprecise way.
 
We first show that the functions in fractional Sobolev space  $W^{s,p}_0(\Omega)$ belong to the ones in $W^{s',1}_0(\Omega)$ for all $1<p<2$ and $0<s'<s<1$.

\begin{lemma}
		\label{lem3.1}
		For a function $u\in W^{s,p}_0(\Omega)$ with $p\in (1,2)$, there holds
		\begin{align*}
			u\in W^{s',1}_0(\Omega)
		\end{align*}
		for all $0<s'<s<1$, moreover,
		\begin{align*}
			[u]_{W^{s',1}(\mathbb{R}^N)}\le C[u]_{W^{s,p}(\mathbb{R}^N)}^p+C,
		\end{align*}
		where the constant $C$ depends on $s,s',\Omega,N$.
\end{lemma}

\begin{proof}
		Observe that  $\mathbb{R}^N\times\mathbb{R}^N$  can be divided into $\Omega\times\Omega_0$, $\Omega_0\times\Omega$, $\Omega\times(\mathbb{R}^N\backslash\Omega_0)$, $(\mathbb{R}^N\backslash\Omega_0)\times\Omega$ and $(\mathbb{R}^N\backslash\Omega)\times(\mathbb{R}^N\backslash\Omega)$, i.e., 
		\begin{align*}
			\mathbb{R}^N\times\mathbb{R}^N=(\Omega\times\Omega_0)\cup(\Omega_0\times\Omega)\cup(\Omega\times(\mathbb{R}^N\backslash\Omega_0))\cup((\mathbb{R}^N\backslash\Omega_0)\times\Omega)\cup((\mathbb{R}^N\backslash\Omega)\times(\mathbb{R}^N\backslash\Omega)),
		\end{align*}
		where $\Omega_0:=\Omega+B_1(0)$. We prove the lemma in different subdomains.
		
		\item[(1)] In $\Omega\times\Omega_0$, from Young's inequality and the similar arguments in \cite[Section 6]{DPV12},  we have
		\begin{align}
			\label{3.1}
			&\quad\int_{\Omega\times\Omega_0}\frac{|u(x)-u(y)|}{|x-y|^{N+s'}}dxdy\notag\\
			&=\int_{(\Omega\times\Omega_0)\cap\left\lbrace |x-y|\le 1\right\rbrace }\frac{|u(x)-u(y)|}{|x-y|^{N+s'}}dxdy
			+\int_{(\Omega\times\Omega_0)\cap\left\lbrace |x-y|>1\right\rbrace }\frac{|u(x)-u(y)|}{|x-y|^{N+s'}}dxdy\notag\\
			&\le \int_{(\Omega\times\Omega_0)\cap\left\lbrace |x-y|\le 1\right\rbrace }\frac{|u(x)-u(y)|}{|x-y|^{N+s'}}dxdy
			+\int_{(\Omega\times\Omega_0)\cap\left\lbrace |x-y|>1\right\rbrace }(|u(x)|+|u(y)|)dxdy\notag\\
			&\le \int_{(\Omega\times\Omega_0)\cap\left\lbrace |x-y|\le 1\right\rbrace }\frac{1}{|x-y|^{N}}\left( \frac{1}{p}\frac{|u(x)-u(y)|^p}{|x-y|^{sp}}+\frac{1}{p'}\frac{1}{|x-y|^{(s'-s)p'}}\right)dxdy\notag\\ &\quad+C(\Omega)\|u\|_{L^1(\Omega)}\notag\\
			&\le \int_{(\Omega\times\Omega_0)\cap\left\lbrace |x-y|\le 1\right\rbrace }\frac{|u(x)-u(y)|^p}{|x-y|^{N+sp}}dxdy+C(\Omega)\|u\|_{L^1(\Omega)}\notag\\
			&\quad+\int_{(\Omega\times\Omega_0)\cap\left\lbrace |x-y|\le 1\right\rbrace }|x-y|^{(s-s')p'-N}dxdy\notag\\
			&\le C(\Omega,N,s)\|u\|_{L^p(\Omega)}+[u]_{W^{s,p}(\mathbb{R}^N)}^p+\int_{\Omega}dx\int_{0}^{1}r^{(s-s')p'-1}dr\notag\\
			&\le C(\Omega,N,s)[u]^p_{W^{s,p}(\mathbb{R}^N)}+C(\Omega,N)\frac {r^{(s-s')p'}}{(s-s')p'}\Bigg|_{0}^{1}\notag\\
			&\le C(\Omega,N,s)[u]^p_{W^{s,p}(\mathbb{R}^N)}+C(\Omega,N,s,s'),
		\end{align}
		where $p'$ is the H\"{o}lder conjugate exponent of $p$, that is $\frac 1p+\frac{1}{p'}=1$.
		
		\item[(2)] In $\Omega\times(\mathbb{R}^N\backslash\Omega_0)$, we have
		\begin{align}
			\label{3.2}
			\int_{\Omega\times(\mathbb{R}^N\backslash\Omega_0)}\frac{|u(x)-u(y)|}{|x-y|^{N+s'}}dxdy&=\int_{\Omega\times(\mathbb{R}^N\backslash\Omega_0)}\frac{|u(x)|}{|x-y|^{N+s'}}dxdy\notag\\
			&\le \int_{\Omega}|u(x)|dx\int_{1}^{+\infty}\frac{1}{r^{s'+1}}dr\notag\\
			&\le C(N,s',\Omega)\|u\|_{L^p(\Omega)}.
		\end{align}
		
		\item[(3)] Similar to (1) and (2) there holds
		\begin{align}
			\label{3.3}
			\int_{(\Omega_0\times\Omega)}\frac{|u(x)-u(y)|}{|x-y|^{N+s'}}dxdy\le C(\Omega)[u]^p_{W^{s,p}(\mathbb{R}^N)}+C(\Omega,N,s,s')
		\end{align}
		and
		\begin{align}
			\label{3.4}
			\int_{(\mathbb{R}^N\backslash\Omega_0)\times\Omega}\frac{|u(x)-u(y)|}{|x-y|^{N+s'}}dxdy\le C(N,s',\Omega)\|u\|_{L^p(\Omega)}.
		\end{align}
		
		\item[(4)] In $(\mathbb{R}^N\backslash\Omega)\times(\mathbb{R}^N\backslash\Omega)$,  from the definition of $W^{s,p}_0(\Omega)$, we get
		\begin{align}
			\label{3.5}
			\int_{(\mathbb{R}^N\backslash\Omega)\times(\mathbb{R}^N\backslash\Omega)}\frac{|u(x)-u(y)|}{|x-y|^{N+s'}}dxdy=0.
		\end{align}\par 
		Therefore, combining \eqref{3.1}--\eqref{3.5}, we obtain
		\begin{align}
			\label{3.6}
			[u]_{W^{s',1}(\mathbb{R}^N)}=\int_{\mathbb{R}^N\times\mathbb{R}^N}\frac{|u(x)-u(y)|}{|x-y|^{N+s'}}dxdy\le C(\Omega,N,s')[u]^p_{W^{s,p}(\mathbb{R}^N)}+C(\Omega,s,s',N).
		\end{align}
		This finishes the proof.
\end{proof}
	
	\begin{remark}
		By using H$\ddot{o}$lder's inequality in \eqref{3.1} rather than Young's inequality, we could show that
		\begin{align*}
			W^{s',1}_0(\Omega)\subset W^{s,p}_0(\Omega)
		\end{align*}
		is continuous. We refer to \cite{C17} for more details about this proof.
	\end{remark}

In the following part of this section, we study the renormalized solutions $\left\lbrace u_p\right\rbrace$ and show the behaviour of the sequence
	\begin{align*}
		\left\lbrace \frac{|u_p(x)-u_p(y)|^{p-2}(u_p(x)-u_p(y))}{|x-y|^{\left( \frac{N}{p}+s\right)(p-1) }}\right\rbrace,
	\end{align*}
and  $\left\lbrace u_p\right\rbrace\subset W^{s',1}_0(\Omega)$ as $p$ goes to $1^+$.

	\begin{proposition}
		\label{pro3.2}
		Let $u_p$ be the renormalized solution to problem \eqref{2.1}, then
		\begin{align*}
			u_p\in L^{\frac{N(p-1)}{N-sp},\infty}(\Omega)
		\end{align*}
		and
		\begin{align*}
			{\rm meas}\left( \left\lbrace 
			(x,y)\in \mathbb{R}^N\times\mathbb{R}^N:\frac{|u_p(x)-u_p(y)|}{|x-y|^{\frac{N}{p}+s}}>h
			\right\rbrace\right)  \le Ch^{\frac{N(1-p)}{N-s}}
		\end{align*}
		for all $h>0$, where $C$ depends only on $f, N, s$.
	\end{proposition}
	\begin{proof}
		From Lemma \ref{lem2.4} and fractional Sobolev's inequality, we obtain
		\begin{align*}
			\|T_k(u_p)\|_{L^{p*}(\Omega)}\le C\left[ T_k(u_p)\right]_{W^{s,p}(\mathbb{R}^N)}\le C\left( Mk\right)^{\frac{1}{p}},  
		\end{align*}
		where $M:=\|f\|_{L^1(\Omega)}$ and $C$ depends only on $N,s,p$ and $p*=\frac{Np}{N-sp}$. Therefore, for any $\varepsilon>0$,
		\begin{align*}
			\text{meas}\left( \left\lbrace x\in\Omega:u_p>\varepsilon\right\rbrace \right) \le \left( \frac{\|T_k(u_p)\|_{L^{p*}(\Omega)}}{\varepsilon}\right)^{p*} \le \left( C(Mk)^\frac{1}{p}\right)^\frac{Np}{N-sp}\varepsilon^{-p*} ,
		\end{align*}
		letting $\varepsilon=k$, we have
		\begin{align*}
			\text{meas}\left( \left\lbrace x\in\Omega:u_p>k\right\rbrace\right)  \le CM^{\frac{N}{N-sp}}k^{\frac{N(1-p)}{N-sp}},
		\end{align*}
		which means, by definition of Lorentz space,
		\begin{align*}
			u_p\in L^{\frac{N(p-1)}{N-sp},\infty}(\Omega).
		\end{align*}\par 
		Subsequently, we show
		\begin{align*}
			\text{meas}\left( \left\lbrace 
			(x,y)\in \mathbb{R}^N\times\mathbb{R}^N:\frac{|u_p(x)-u_p(y)|}{|x-y|^{\frac{N}{p}+s}}>h
			\right\rbrace\right)  \le Ch^{\frac{N(1-p)}{N-s}}.
		\end{align*}
		Similar to the arguments used in \cite{BGD93}, set
		\begin{align*}
			\varPhi(k,\lambda)=\text{meas}\left( \left\lbrace (x,y)\in \mathbb{R}^N\times\mathbb{R}^N:\frac{|u_p(x)-u_p(y)|^p}{|x-y|^{N+sp}}>\lambda,\  \max\left\lbrace u_p(x),u_p(y)\right\rbrace>k \right\rbrace\right) .
		\end{align*}
		Obviously, we have
		\begin{align}
			\label{3.8}
			\varPhi(k,0)\le CM^{\frac{N}{N-sp}}k^{\frac{N(1-p)}{N-sp}}
		\end{align}
		and $\lambda\mapsto\varPhi(k,\lambda)$ decreasing in $(0,+\infty)$. In addition, note that
		\begin{align*}
			\frac{1}{\lambda}\int_{0}^{\lambda}\left( \varPhi(0,\mu)-\varPhi(k,0)\right)d\mu\le \frac{1}{\lambda}\int_{0}^{\lambda} \left( \varPhi(0,\mu)-\varPhi(k,\mu)\right)d\mu,
		\end{align*}
		there holds
		\begin{align}
			\label{3.9}
				\varPhi(0,\lambda)&\le \frac{1}{\lambda}\int_{0}^{\lambda}\varPhi(0,\mu)d\mu\notag\\
				&\le \varPhi(k,0)+\frac{1}{\lambda}\int_{0}^{\lambda}\left( \varPhi(0,\mu)-\varPhi(k,\mu)\right)d\mu\notag\\
				&=:\varPhi(k,0)+I. 
		\end{align}
		For $I$, we have
		\begin{align}
			\label{3.10}
				I&=\int_{0}^{\lambda}\text{meas}\left( \left\lbrace (x,y)\in\mathbb{R}^N\times\mathbb{R}^N:\frac{|u_p(x)-u_p(y)|^p}{|x-y|^{N+sp}}>\mu,\ \max\left\lbrace u_p(x),u_p(y)\right\rbrace\le k \right\rbrace\right)  d\mu\notag\\
				&\le \int_{\left\lbrace |u_p|\le k\right\rbrace\times \left\lbrace |u_p|\le k\right\rbrace}\frac{|u_p(x)-u_p(y)|^p}{|x-y|^{N+sp}}dxdy\notag\\
				&\le \int_{\mathbb{R}^N\times\mathbb{R}^N}\frac{|T_k(u_p)(x)-T_k(u_p)(y)|^p}{|x-y|^{N+sp}}dxdy\notag\\
				&\le kM.
		\end{align}
		Combining \eqref{3.8}, \eqref{3.9} and \eqref{3.10}, we get
		\begin{align*}
			\varPhi(0,\lambda)\le \frac{kM}{\lambda}+CM^{\frac{N}{N-sp}}k^{\frac{N(1-p)}{N-sp}}
		\end{align*}
		for all $k>0$. Therefore,
		\begin{align*}
			\varPhi(0,\lambda)&\le \min\limits_{k>0}\left\lbrace \frac{kM}{\lambda}+CM^{\frac{N}{N-sp}}k^{\frac{N(1-p)}{N-sp}}\right\rbrace\\
			&=C\lambda^{\frac{N(1-p)}{p(N-s)}},
		\end{align*}
		where $C$ is a constant depending only on $N,s,M,\Omega$. Letting $\lambda=h^p$, we obtain
		\begin{align*}
			\text{meas}\left( \left\lbrace (x,y)\in \mathbb{R}^N\times\mathbb{R}^N:\frac{|u_p(x)-u_p(y)|}{|x-y|^{\frac{N}{p}+s}}>h\right\rbrace\right) \le Ch^{\frac{N(1-p)}{N-s}}. 
		\end{align*}
	\end{proof}
	\begin{remark}
		\label{rem3.3}
		We point out that, in Proposition \ref{pro3.2}, although the constant $C$ depends on $p$, we could limit $p\in (1,2)$, and we may choose a constant $C_0$ big enough, depending only on $N,s,f,\Omega$, such that for all $C=C(N,p,s,\Omega,f)$, $C<C_0$.
	\end{remark}

	\begin{theorem}
		\label{the3.4}
		Let $u_p$ be the renormalized solution to problem \eqref{2.1}, there exists a measurable function $u$ and a function $Z$, such that
		\begin{align*}
			u_p\rightarrow u\quad \text{a.e. in }\Omega,
		\end{align*}
		\begin{align*}
			\frac{|u_p(x)-u_p(y)|^{p-2}\left( u_p(x)-u_p(y)\right) }{|x-y|^{\left( \frac{N}{p}+s\right)(p-1) }}\rightharpoonup Z(x,y) \quad \text{in } L^q(A_\Omega)
		\end{align*}
		and
		\begin{align*}
			|u_p(\cdot)|^{p-2}u_p(\cdot)\rightharpoonup Z(\cdot, y) \quad \text{in }L^q(\Omega)
		\end{align*}
		for any fixed $y\in \left( \mathbb{R}^N\backslash\Omega_0\right)$ and $q\in \left[ 1,\frac{N}{N-s}\right) $, $A_\Omega=(\Omega\times\Omega_0)\cup(\Omega_0\times\Omega)$, $\Omega_0=\Omega+B_1(0).$
	\end{theorem}

	\begin{proof}
		Define
		\begin{align*}
			S_n(t)=
			\begin{cases}
				1  &\text{if } t\le n,\\
				n-t+1 &\text{if }  n<t\le n+1,\\
				0 & \text{if } t>n
			\end{cases}
		\end{align*}
		and
		\begin{align*}
			\varphi(t)=\frac{t}{1+t}.
		\end{align*}
		We have
		\begin{align*}
			\varphi(u_p)\in W_0^{s,p}(\Omega)\cap L^{\infty}(\Omega).
		\end{align*}
		Therefore,
		\begin{align*}
			S_n(u_p)\varphi(u_p)\in W^{s,p}_0(\Omega)
		\end{align*}
		and
		\begin{align}
			\label{3.11}
				&\quad\int_{\mathcal{C}_\Omega}\frac{|u_{p}(x)-u_{p}(y)|^{p-2}\left( u_{p}(x)-u_{p}(y)\right)\left[ \left( S_n(u_{p})\varphi(u_p)\right) (x)-\left( S_n(u_{p})\varphi(u_p)\right) (y)\right]}{|x-y|^{N+sp}}dxdy\notag\\
				&=\int_{\Omega}fS_n(u_p)\varphi(u_p)dx.
		\end{align}
	
		Note that
		\begin{align*}
			\left( S_n(u_{p})\varphi(u_p)\right) (x)-\left( S_n(u_{p})\varphi(u_p)\right) (y)&=\left[ S_n(u_{p})(x)-S_n(u_{p})(y)\right]\frac{\varphi(u_p)(x)+\varphi(u_p)(y)}{2}\\
			&\ \quad+\left[ \varphi(u_p)(x)-\varphi(u_p)(y)\right] \frac{S_n(u_{p})(x)+S_n(u_{p})(y)}{2},
		\end{align*}
		we can divide  the left-hand side of \eqref{3.11} into two parts, by denoting
		\begin{align*}
			\begin{aligned}
				I_1=\int_{\mathcal{C}_\Omega}&\frac{|u_{p}(x)-u_{p}(y)|^{p-2}\left( u_{p}(x)-u_{p}(y)\right)\left[  S_n(u_{p}) (x)-S_n(u_{p})(y)\right]}{|x-y|^{N+sp}}\\
				&\times \frac{\varphi(u_p)(x)+\varphi(u_p)(y)}{2}dxdy
			\end{aligned}
		\end{align*}
		and
		\begin{align*}
			\begin{aligned}
				I_2=\int_{\mathcal{C}_\Omega}&\frac{|u_{p}(x)-u_{p}(y)|^{p-2}\left( u_{p}(x)-u_{p}(y)\right)\left[  \varphi(u_{p}) (x)-\varphi(u_{p})(y)\right]}{|x-y|^{N+sp}}\\
				&\times\frac{S_n(u_p)(x)+S_n(u_p)(y)}{2}dxdy.
			\end{aligned}
		\end{align*}
		By Lemma \ref{lem2.3} and $I_1+I_2\le \|f\|_{L^1(\Omega)}$, we get
		\begin{align*}
				\lim\limits_{n\rightarrow+\infty}I_2\le\|f\|_{L^1(\Omega)},
		\end{align*}
		that is
		\begin{align*}
			\int_{\mathcal{C}_\Omega}\frac{|u_{p}(x)-u_{p}(y)|^{p-2}\left( u_{p}(x)-u_{p}(y)\right)\left[  \varphi(u_{p}) (x)-\varphi(u_{p})(y)\right]}{|x-y|^{N+sp}} dxdy
			\le\|f\|_{L^1(\Omega)}.
		\end{align*}
		Since
		\begin{align*}
			\left| \varphi(u_p)(x)-\varphi(u_p)(y)\right| \le |u_p(x)-u_p(y)|
		\end{align*}
		and
		\begin{align*}
			\left( \varphi(u_p)(x)-\varphi(u_p)(y)\right) (u_p(x)-u_p(y))\ge 0,
		\end{align*}
		we arrive at
		\begin{align*}
			\int_{\mathcal{C}_\Omega}\frac{\left| \varphi(u_p)(x)-\varphi(u_p)(y)\right|^p}{|x-y|^{N+sp}}dxdy\le \|f\|_{L^1(\Omega)}.
		\end{align*}
		Taking into account $\left[ \varphi(u_p)\right]_{W^{s',1}(\mathbb{R}^N)} \le  C\left[ \varphi(u_p)\right]^p_{W^{s,p}(\mathbb{R}^N)}+C $ for any fixed $s'\in(0,s)$, we have $\left\lbrace \varphi(u_p)\right\rbrace $ is bounded in $W^{s',1}_0(\Omega)$. As a consequence, up to a subsequence, still denote by $\left\lbrace \varphi(u_p)\right\rbrace $,
		\begin{align*}
			\varphi(u_p)\rightarrow\varphi(u) \quad \text{in } L^1(\Omega)\text{ and a.e. in }\mathbb{R}^N.
		\end{align*}
		Recalling $\varphi(t)$ is strictly increasing, the sequence $\left\lbrace u_p\right\rbrace $ tends to a measurable function $u$ for almost all $x\in \Omega$. In particular, if $\varphi(u)=1$, we set $u=+\infty$.
		 
		Next, we try to find a function $Z$, defined on $\mathbb{R}^N\times\mathbb{R}^N$. Note that Proposition \ref{pro3.2} and Remark \ref{rem3.3} imply that
		\begin{align*}
			\left( 
			\int_{A_\Omega}\left( \frac{|u_p(x)-u_p(y)|^{p-1}}{|x-y|^{\left( \frac{N}{p}+s\right)(p-1) }}\right)^qdx 
			\right) ^\frac{1}{q}\le C,\quad q\in \left[ 1,\frac{N}{N-s}\right)   ,
		\end{align*}
		where $A_\Omega=(\Omega\times\Omega_0)\cup(\Omega_0\times\Omega)$ and $C$ depends only on $N,s,\Omega$. Thereby, up to a subsequence, still denote by
		\begin{align*}
			\left\lbrace 
			\frac{|u_p(x)-u_p(y)|^{p-2}(u_p(x)-u_p(y))}{|x-y|^{\left( \frac{N}{p}+s\right)(p-1) }}
			\right\rbrace ,
		\end{align*}
		we have
		\begin{align*}
			\frac{|u_p(x)-u_p(y)|^{p-2}(u_p(x)-u_p(y))}{|x-y|^{\left( \frac{N}{p}+s\right)(p-1) }}\rightharpoonup Z_q(x,y) \quad \text{in }L^q(A_\Omega).
		\end{align*}
		By a diagonal argument, we find the limit $Z_q$ does not depend on $q$. Hence, we write $Z_q$ as $\bar{Z}$, defined on $A_\Omega$.
		
		Similarly, we could find a function $\hat{Z}(x)$, such that for fixed $y\in \mathbb{R}^N\backslash\Omega_0$,
		\begin{align*}
			|u_p(x)|^{p-2}u_p(x)\rightharpoonup \hat{Z}(x) \quad \text{in } L^q(\Omega), \quad q\in \left[ 1,\frac{N}{N-s}\right),
		\end{align*}
		where we utilized $u_p\in L^{\frac{N(p-1)}{N-sp},\infty}(\Omega)$.
		
		By setting
		\begin{align*}
			Z(x,y)=\begin{cases}
				\bar{Z}(x,y) &\text{if }(x,y)\in A_\Omega,\\[1mm]
				\hat{Z}(x) &\text{if }(x,y)\in \Omega\times(\mathbb{R}^N\backslash\Omega_0),\\[1mm]
				-\hat{Z}(y) &\text{if }(x,y)\in (\mathbb{R}^N\backslash\Omega_0)\times\Omega,\\[1mm]
				0 &\text{if }(x,y)\in (\mathbb{R}^N\times\mathbb{R}^N)\backslash\mathcal{C}_\Omega,
			\end{cases}
		\end{align*}
		we then complete the proof.
	\end{proof} 
	\begin{remark}
		In the proof of Theorem \ref{the3.4}, since $p>1$ and $u_p\in L^{\frac{N(p-1)}{N-sp},\infty}(\Omega)$, we have $\left\lbrace u_p\right\rbrace $ is bounded in $L^{q}(\Omega)$ for all $1\le q<\frac{N}{N-s}<\frac{N}{N-sp}$.
	\end{remark}

\section{The limit problem}
\label{sec4}
In this section, we show the connection between the function $u,Z$ and the ``limit problem" of \eqref{2.1}, which is
	\begin{align}
		\label{4.1}
		\begin{cases}
			\mathcal{L}_1u=f &\text{in }\Omega,\\
			u=0 &\text{in }\mathbb{R}^N\backslash\Omega.
		\end{cases}
	\end{align}

\medskip

\noindent \textit{Proof of Theorem \ref{the1.2}.}
		We prove Theorem \ref{the1.2} by several steps. Some of the reasoning is based on the ideas developed in \cite{MST08,MY23}.
		
		\smallskip
		
		\textbf{Step 1.}  We shall prove that $T_k(u)\in W^{s,1}_0(\Omega)$ for any fixed $k>0$.
		
		For fixed $k>0$, we have
		\begin{align*}
			\left[ T_k(u_p)\right]_{W^{s,p}(\mathbb{R}^N)}^p\le k\|f\|_{L^1(\Omega)}.
		\end{align*}
		Note that $T_k(u_p)\rightarrow T_k(u) \text{ a.e. in }\Omega$ and from Fatou's lemma, there holds
		\begin{align*}
			[T_k(u)]_{W^{s,1}(\mathbb{R}^N)}\le \lim\limits_{p\rightarrow1^+}[T_k(u_p)]_{W^{s,p}(\mathbb{R}^N)}^p\le k\|f\|_{L^1(\Omega)},
		\end{align*}
		where Fatou's lemma is utilized. Moreover, $u_p=0$ a.e. in $\mathbb{R}^N\backslash\Omega$ implies that $u=0\text{ a.e. in }\mathbb{R}^N\backslash\Omega$. Therefore, it follows that $T_k(u)\in W^{s,1}_0(\Omega)$.
		
				\smallskip
				
		\textbf{Step 2.} We shall prove that $Z_k\in L^\infty(\mathbb{R}^N\times\mathbb{R}^N)$ and $\|Z_k\|_{L^\infty(\mathbb{R}^N\times\mathbb{R}^N)}\le 1$.
		
		For any fixed $k>0$, the sequence
		\begin{align*}
			\left\lbrace 
			\frac{|u_p(x)-u_p(y)|^{p-2}(u_p(x)-u_p(y))}{|x-y|^{\left( \frac{N}{p}+s\right)(p-1) }}\mathcal{X}_{\left\lbrace (x,y)\in A_\Omega:u_p(x)<k\text{ and }u_p(y)<k\right\rbrace }
			\right\rbrace 
		\end{align*}
		is bounded in $L^q(A_\Omega), q\in \left[ 1,\frac{N}{N-s}\right)  $. Thus, as $p\rightarrow1^+$, we have, up to a subsequence,
		\begin{align*}
			\frac{|u_p(x)-u_p(y)|^{p-2}(u_p(x)-u_p(y))}{|x-y|^{\left( \frac{N}{p}+s\right)(p-1) }}\mathcal{X}_{\left\lbrace (x,y)\in A_\Omega:u_p(x)<k\text{ and }u_p(y)<k\right\rbrace }\rightharpoonup w_k
		\end{align*}
		in $L^1(A_\Omega)$.
		
		Denote
		\begin{align*}
			B_{p,h,k}=\left\lbrace 
			(x,y)\in A_\Omega:\frac{|T_k(u_p)(x)-T_k(u_p)(y)|}{|x-y|^{\frac{N}{p}+s}}>h
			\right\rbrace.
		\end{align*}
		Similar to the same arguments as before, we also have
		\begin{align*}
			\frac{|u_p(x)-u_p(y)|^{p-2}(u_p(x)-u_p(y))}{|x-y|^{\left( \frac{N}{p}+s\right)(p-1) }}\mathcal{X}_{B_{p,h,k}\cap\left\lbrace (x,y)\in A_\Omega:u_p(x)<k\text{ and }u_p(y)<k\right\rbrace }\rightharpoonup g_{h,k}
		\end{align*}
		and
		\begin{align*}
			\frac{|u_p(x)-u_p(y)|^{p-2}(u_p(x)-u_p(y))}{|x-y|^{\left( \frac{N}{p}+s\right)(p-1) }}\mathcal{X}_{\left( A_\Omega\backslash B_{p,h,k}\right) \cap\left\lbrace (x,y)\in A_\Omega:u_p(x)<k\text{ and }u_p(y)<k\right\rbrace }\rightharpoonup f_{h,k}
		\end{align*}
		in $L^1(A_\Omega)$ for some $g_{h,k}$ and $f_{h,k}$.
		 
		By the definition of $B_{p,h,k}$, the following inequality holds,
		\begin{align*}
			\text{meas}\left(  B_{p,h,k}\right)  \le \frac{1}{h^p}\int_{A_\Omega}\frac{|T_k(u_p)(x)-T_k(u_p)(y)|^p}{|x-y|^{N+sp}}dxdy\le \frac{k\|f\|_{L^1(\Omega)}}{h^p}.
		\end{align*}
		Therefore, from H\"{o}lder's inequality, we have
		\begin{align}
			\label{4.3}
				&\quad\left| \int_{B_{p,h,k}\cap\left\lbrace (x,y)\in A_\Omega:u_p(x)<k\text{ and }u_p(y)<k\right\rbrace}\frac{|u_p(x)-u_p(y)|^{p-2}(u_p(x)-u_p(y))}{|x-y|^{\left( \frac{N}{p}+s\right)(p-1) }}dxdy\right| \notag\\
				&\le \left( \int_{A_\Omega}\left[ \frac{|T_k(u_p)(x)-T_k(u_p)(y)|^{p-2}\left( T_k(u_p)(x)-T_k(u_p)(y)\right) }{|x-y|^{\left( \frac{N}{p}+s\right)(p-1) }}\right]^\frac{p}{p-1}dxdy \right)^{\frac{p-1}{p}}\notag\\
				&\quad\ \times\left( \text{meas}\left\lbrace B_{p,h,k}\right\rbrace \right)^\frac{1}{p} \notag\\
				&\le \left( k\|f\|_{L^1(\Omega)}\right)^\frac{p-1}{p}\left( \frac{k\|f\|_{L^1(\Omega)}}{h^p}\right) ^\frac{1}{p}\le \frac{k\|f\|_{L^1(\Omega)}}{h^p}.
		\end{align}
		By letting $p$ goes to $1^+$, for fixed $k,h>0$, it follows from \eqref{4.3} 
		\begin{align*}
			\left| \int_{A_\Omega}g_{h,k}\phi dxdy\right|\le  \frac{k\|f\|_{L^1(\Omega)}}{h}
		\end{align*}
		for any $\phi\in L^\infty(A_\Omega)$ with $\|\phi\|_{L^\infty(A_\Omega)}\le 1$, hence we deduce
		\begin{align*}
			\int_{A_\Omega}\left| g_{h,k}\right|dxdy\le  \frac{k\|f\|_{L^1(\Omega)}}{h}.
		\end{align*}
		
		On the other hand, we have
		\begin{align*}
			\frac{|u_p(x)-u_p(y)|^{p-2}(u_p(x)-u_p(y))}{|x-y|^{\left( \frac{N}{p}+s\right)(p-1) }}\mathcal{X}_{\left( A_\Omega\backslash B_{p,h,k}\right) \cap\left\lbrace (x,y)\in A_\Omega:u_p(x)<k\text{ and }u_p(y)<k\right\rbrace }\le h^{p-1}
		\end{align*}
		a.e. in $A_\Omega$. This inequality implies the following pointwise estimate for $f_{h,k}$,
		\begin{align*}
			\left| f_{h,k}\right|\le \lim\limits_{p\rightarrow1}h^{p-1}=1\quad\text{a.e. in }A_\Omega. 
		\end{align*}
		From
		\begin{align*}
			w_k=f_{h,k}+g_{h,k}
		\end{align*}
		with
		\begin{align*}
			\|f_{h,k}\|_{L^\infty(A_\Omega)}\le 1,\quad \int_{A_\Omega}\left| g_{h,k}\right|dxdy\le \frac{k\|f\|_{L^1(\Omega)}}{h} ,
		\end{align*}
		we obtain
		\begin{align*}
			\lim\limits_{h\rightarrow+\infty}g_{h,k}=0
		\end{align*}
		and
		\begin{align*}
			\lim\limits_{h\rightarrow+\infty}f_{h,k}=\lim\limits_{h\rightarrow+\infty}(w_k-g_{h,k})=w_k.
		\end{align*}
		Since $\|f_{h,k}\|_{L^\infty(A_\Omega)}\le 1$, we obtain $w_k\in L^\infty(A_\Omega)$ and $\|w_k\|_{L^\infty(A_\Omega)}\le 1$. Next we show the connection between $w_k$ and $Z_k$. Taking $u_p\rightarrow u\text{ a.e. in }\Omega$ into account, we deduce
		\begin{align*}
			\mathcal{X}_{\left\lbrace (x,y)\in A_\Omega:u_p(x)<k\text{ and }u_p(y)<k\right\rbrace }\rightarrow\mathcal{X}_{\left\lbrace (x,y)\in A_\Omega\text{ and }u(y)<k\right\rbrace }
		\end{align*}
		in $L^\rho(A_\Omega)$ for all $\rho\in \left[ 1,+\infty\right) $. By means of $\text{meas}\left(  A_\Omega\right) <+\infty$, the set of $k$ such that
		\begin{align*}
			\text{meas}\left( \left\lbrace (x,y)\in A_\Omega:u(x)=k\text{ or }u(y)=k\right\rbrace\right) >0 
		\end{align*}
		is countable, which means, for almost all $k>0$, there holds $w_k=Z_k$. Observing that
		\begin{align*}
			\lim\limits_{k\rightarrow+\infty}w_k=\lim\limits_{k\rightarrow+\infty}Z_k=Z_\infty=Z\mathcal{X}_{\left\lbrace (x,y)\in A_\Omega:u(x)<+\infty\text{ and }u(y)<+\infty\right\rbrace }
		\end{align*}
		holds almost everywhere in $A_\Omega$, and combining with $\|w_k\|_{L^\infty(A_\Omega)}\le 1$, we have
		\begin{align*}
			\|Z_k\|_{L^\infty(A_\Omega)}\le \|Z_\infty\|_{L^\infty(A_\Omega)}\le 1
		\end{align*}
		for all $k>0$. Moreover, we have
		\begin{align}
			\label{4.4}
			Z_k\rightharpoonup Z_\infty\quad\text{in }L^q(A_\Omega)
		\end{align}
		for any $q\in [1,+\infty)$.
		
		Next, we study the property of $Z$ on $\mathcal{C}_\Omega\backslash A_\Omega$. From Theorem \ref{the3.4}, we know for any fixed $y\in \mathbb{R}^N\backslash\Omega_0$,
		\begin{align*}
			|u_p(\cdot)|^{p-2}u_p(\cdot)\rightharpoonup Z(\cdot,y)\quad\text{in }L^q(\Omega),\quad q\in \left[ 1,\frac{N}{N-s}\right) .
		\end{align*}
		Similar to the discussion of $Z$ in $A_\Omega$, we arrive at
		\begin{align*}
			Z_k(\cdot,y)\in L^\infty(\Omega),\quad \|Z_k(\cdot,y)\|_{L^\infty(\Omega)}\le 1
		\end{align*}
		and
		\begin{align}
			\label{4.5}
			Z_k(\cdot,y)\rightharpoonup Z_\infty(\cdot,y)\quad\text{in }L^q(\Omega)
		\end{align}
		for any $y\in\mathbb{R}^N\backslash\Omega_0$, $q\in [1,+\infty)$ and fixed $k>0$.
		Therefore,
		\begin{align*}
			Z_k(x,y)\in L^\infty\left( \mathcal{C}_\Omega\backslash A_\Omega\right) ,\quad \|Z_k(x,y)\|_{L^\infty\left( \mathcal{C}_\Omega\backslash A_\Omega\right)}\le 1.
		\end{align*}
	
			\smallskip
		\textbf{Step 3.} We shall prove that for fixed $k>0$, there holds
		\begin{align*}
			Z\left( u(x)-u(y)\right)=\left| u(x)-u(y)\right| \quad \text{a.e. in}  \left\lbrace (x,y)\in\mathbb{R}^N\times\mathbb{R}^N:u(x)<k\text{ and }u(y)<k\right\rbrace .
		\end{align*}
		
		Firstly, we discuss it in $A_\Omega$. By Step 2, $\|Z_k\|_{L^\infty(A_\Omega)}\le 1$, hence
		\begin{align*}
			Z_k(T_k(u)(x)-T_k(u)(y))\le |T_k(u)(x)-T_k(u)(y)|.
		\end{align*}
		Now we prove
		\begin{align*}
			|T_k(u)(x)-T_k(u)(y)|\le Z_k(T_k(u)(x)-T_k(u)(y)).
		\end{align*}
		For fixed $k>0$ and any subset $E\subset A_\Omega$ with $\text{meas}\left(   E\right) >0 $, because of
		\begin{align*}
			T_k(u_p)\rightarrow T_k(u)\quad \text{a.e. in }\Omega
		\end{align*}
		and
		\begin{align*}
			\|T_k(u_p)(x)-T_k(u_p)(y)\|_{L^\infty(A_\Omega)}\le 2k,
		\end{align*}
		we have
		\begin{align*}
			T_k(u_p)(x)-T_k(u_p)(y)\rightarrow T_k(u)(x)-T_k(u)(y)\quad \text{in }L^\rho(A_\Omega)
		\end{align*}
		for all $\rho\in \left[ 1,+\infty\right) $, by Fatou's lemma, we have
		\begin{align*}
			&\quad \int_{E\cap\left\lbrace (x,y)\in A_\Omega:u(x)<k\text{ and }u(y)<k\right\rbrace }|T_k(u)(x)-T_k(u)(y)|dxdy\\
			&\le \lim\limits_{p\rightarrow1}\int_{E\cap\left\lbrace (x,y)\in A_\Omega:u(x)<k\text{ and }u(y)<k\right\rbrace }\frac{|T_k(u_p)(x)-T_k(u_p)(y)|^p}{|x-y|^{\left( \frac{N}{p}+s\right)(p-1) }}dxdy\\
			&\le \lim\limits_{p\rightarrow1}\int_{E\cap\left\lbrace (x,y)\in A_\Omega:u(x)<k\text{ and }u(y)<k\right\rbrace }\frac{|u_p(x)-u_p(y)|^p\mathcal{X}_{\left\lbrace (x,y)\in A_\Omega:u_p(x)<k\text{ and }u_p(y)<k\right\rbrace }}{|x-y|^{\left( \frac{N}{p}+s\right)(p-1) }}dxdy\\
			&\quad+\lim\limits_{p\rightarrow1}\int_{E\cap\left\lbrace (x,y)\in A_\Omega:u(x)<k\text{ and }u(y)<k\right\rbrace }\frac{|u_p(x)-k|^p\mathcal{X}_{\left\lbrace (x,y)\in A_\Omega:u_p(x)<k\text{ and }u_p(y)\ge k\right\rbrace }}{|x-y|^{\left( \frac{N}{p}+s\right)(p-1) }}dxdy\\
			&\quad+\lim\limits_{p\rightarrow1}\int_{E\cap\left\lbrace (x,y)\in A_\Omega:u(x)<k\text{ and }u(y)<k\right\rbrace }\frac{|k-u_p(y)|^p\mathcal{X}_{\left\lbrace (x,y)\in A_\Omega:u_p(x)\ge k\text{ and }u_p(y)<k\right\rbrace }}{|x-y|^{\left( \frac{N}{p}+s\right)(p-1) }}dxdy.\\
		\end{align*}
		Since
		\begin{align*}
			&\mathcal{X}_{\left\lbrace (x,y)\in A_\Omega:u_p(x)<k\text{ and }u_p(y)<k\right\rbrace }\rightarrow\mathcal{X}_{\left\lbrace (x,y)\in A_\Omega:u(x)<k\text{ and }u(y)<k\right\rbrace },\\
			&\mathcal{X}_{\left\lbrace (x,y)\in A_\Omega:u_p(x)<k\text{ and }u_p(y)\ge k\right\rbrace}\rightarrow\mathcal{X}_{\left\lbrace (x,y)\in A_\Omega:u(x)<k\text{ and }u(y)\ge k\right\rbrace },\\
			&\mathcal{X}_{\left\lbrace (x,y)\in A_\Omega:u_p(x)\ge k\text{ and }u_p(y)<k\right\rbrace}\rightarrow\mathcal{X}_{\left\lbrace (x,y)\in A_\Omega:u(x)\ge k\text{ and }u(y)<k\right\rbrace }
		\end{align*}
		in $L^\rho(A_\Omega)$ for $\rho\in \left[ 1,+\infty\right) $, we have
		\begin{align*}
			&\quad \int_{E\cap\left\lbrace (x,y)\in A_\Omega:u(x)<k\text{ and }u(y)<k\right\rbrace }|u(x)-u(y)|dxdy\\
			&\le \int_{E\cap\left\lbrace (x,y)\in A_\Omega:u(x)<k\text{ and }u(y)<k\right\rbrace }Z\left( u(x)-u(y)\right)dxdy
		\end{align*}
		and complete the discussion in $A_\Omega$.
		
		Next, we discuss in $\mathcal{C}_\Omega\backslash A_\Omega$.	We have $\|Z_k\|_{L^\infty(\mathcal{C}_\Omega\backslash A_\Omega)}\le 1$, which implies
		\begin{align*}
			Z_k(T_k(u)(x)-T_k(u)(y))\le |T_k(u)(x)-T_k(u)(y)|.
		\end{align*}
		Therefore, it is enough for us to prove
		\begin{align*}
			|T_k(u)(x)-T_k(u)(y)|\le Z_k(T_k(u)(x)-T_k(u)(y))\quad\text{in}\ \mathcal{C}_\Omega\backslash A_\Omega,
		\end{align*}
		which is equivalent to
		\begin{align*}
			|T_k(u)(x)|\le Z_kT_k(u)(x)\quad\text{in }\Omega
		\end{align*}
		for fixed $y\in \mathbb{R}^N\backslash\Omega_0$ because of $u=0$ a.e. in $\mathbb{R}^N\backslash\Omega$.
		 
		Similar to the discussion in $A_\Omega$, we have, for any subset $E\subset\Omega$ with $\text{meas}\left(   E\right) >0 $,
		\begin{align*}
			&\quad\int_{E\cap\left\lbrace x\in\Omega:u(x)<k\right\rbrace }|u(x)|dx\\
			&\le \lim\limits_{p\rightarrow1}\int_{E\cap\left\lbrace x\in\Omega:u(x)<k\right\rbrace }|T_k(u_p)(x)|^pdx\\
			&\le \lim\limits_{p\rightarrow1}\int_{E\cap\left\lbrace x\in\Omega:u(x)<k\right\rbrace }|u_p(x)|^p\mathcal{X}_{\left\lbrace x\in\Omega:u_p(x)<k\right\rbrace }dx\\
			&\quad +\lim\limits_{p\rightarrow1}\int_{E\cap\left\lbrace x\in\Omega:u(x)<k\right\rbrace }k^p\mathcal{X}_{\left\lbrace x\in\Omega:u_p(x)\ge k\right\rbrace }dx\\
			&=\int_{E\cap\left\lbrace x\in\Omega:u(x)<k\right\rbrace }Zudx.
		\end{align*}
		Thus we arrive at (3) of the Definition \ref{def1.1}.
		
				\smallskip
				
		\textbf{Step 4.} We shall prove that $Z$ is an anti-symmetric function.
		
		We only prove that $Z$ is an anti-symmetric function in $A_\Omega$. From Theorem \ref{the3.4}, we have
		\begin{align*}
			\frac{|u_p(x)-u_p(y)|^{p-2}(u_p(x)-u_p(y))}{|x-y|^{\left( \frac{N}{p}+s\right)(p-1) }}\rightharpoonup Z(x,y) \quad \text{in }L^1(A_\Omega).
		\end{align*}
		On the other hand, exchange $x$ and $y$, there also holds
		\begin{align*}
			-\frac{|u_p(x)-u_p(y)|^{p-2}(u_p(x)-u_p(y))}{|x-y|^{\left( \frac{N}{p}+s\right)(p-1) }}\rightharpoonup -Z(x,y) \quad \text{in }L^1(A_\Omega)
		\end{align*}
		and
		\begin{align*}
			\frac{|u_p(y)-u_p(x)|^{p-2}(u_p(y)-u_p(x))}{|x-y|^{\left( \frac{N}{p}+s\right)(p-1) }}\rightharpoonup Z(y,x) \quad\text{in }L^1(A_\Omega),
		\end{align*}
		hence
		\begin{align*}
			Z(x,y)=-Z(y,x) \quad \text{in }A_\Omega
		\end{align*}
		because of the anti-symmetry of the function
		\begin{align*}
			\frac{|u_p(x)-u_p(y)|^{p-2}(u_p(x)-u_p(y))}{|x-y|^{\left( \frac{N}{p}+s\right)(p-1) }}.
		\end{align*}
	
			\smallskip
			
		\textbf{Step 5.} We shall prove that \eqref{4.2} holds under some suitable assumptions.
		
		We begin from the equality
		\begin{align}
			\label{4.6}
				&\quad\int_{\mathcal{C}_\Omega}\frac{|u_p(x)-u_p(y)|^{p-2}\left( u_p(x)-u_p(y)\right)\left[ \left( S(u_p)\varphi\right)(x)-\left( S(u_p)\varphi\right)(y) \right]  }{|x-y|^{N+sp}}dxdy\notag\\
				&=\int_{\Omega}fS(u_p)\varphi dx
		\end{align}
		with $S(t)\in W^{1,\infty}(\mathbb{R})$ is monotone in $(0,+\infty)$, having compact support and $\varphi\in W^{1,\infty}(\Omega)$, satisfying
		\begin{align*}
			S(u_p)\varphi\in W^{s,p}_0(\Omega),\ S'(t)\varphi(x)\ge 0\text{ for all }t\in [0,\infty)\text{ and }x\in \Omega.
		\end{align*}
		From $u_p\to u$ a.e. in $\Omega$, we have, the right-hand side of \eqref{4.6} satisfies
		\begin{align}
			\label{4.7}
			\lim\limits_{p\rightarrow1}\int_{\Omega}fS(u_p)\varphi dx=\int_{\Omega}fS(u)\varphi dx.
		\end{align}
		Next, we discuss the left-hand side of \eqref{4.6}. Set
		\begin{align*}
			I_1:=\int_{\mathcal{C}_\Omega}\frac{|u_p(x)-u_p(y)|^{p-2}\left( u_p(x)-u_p(y)\right)\left( \varphi(x)-\varphi(y) \right)   }{|x-y|^{N+sp}}\frac{S(u_p)(x)+S(u_p)(y)}{2}dxdy
		\end{align*}
		and
		\begin{align*}
			I_2:=\int_{\mathcal{C}_\Omega}\frac{|u_p(x)-u_p(y)|^{p-2}\left( u_p(x)-u_p(y)\right)\left( S(u_p)(x)-S(u_p)(y) \right)   }{|x-y|^{N+sp}}\frac{\varphi(x)+\varphi(y)}{2}dxdy.
		\end{align*}
		For $I_1$, assuming $\text{supp}\,S\subset[-k,k]$ and dividing $A_\Omega$ into the following nine subdomains,
		\begin{align*}
			&E_{1,p}=\left\lbrace (x,y)\in A_\Omega:u_p(x)<k \text{ and } u_p(y)<k\right\rbrace,\\
			&E_{2,p}=\left\lbrace (x,y)\in A_\Omega:u_p(x)<k \text{ and } u_p(y)\ge k+1\right\rbrace,\\
			&E_{3,p}=\left\lbrace (x,y)\in A_\Omega:u_p(x)<k \text{ and } k\le u_p(y)< k+1\right\rbrace,\\
			&E_{4,p}=\left\lbrace (x,y)\in A_\Omega:k\le u_p(x)<k+1 \text{ and } u_p(y)< k\right\rbrace,\\
			&E_{5,p}=\left\lbrace (x,y)\in A_\Omega:k\le u_p(x)<k+1 \text{ and } k\le u_p(y)< k+1\right\rbrace,\\
			&E_{6,p}=\left\lbrace (x,y)\in A_\Omega:k\le u_p(x)<k+1 \text{ and } u_p(y)\ge k+1\right\rbrace,\\
			&E_{7,p}=\left\lbrace (x,y)\in A_\Omega:u_p(x)\ge k+1 \text{ and } u_p(y)< k\right\rbrace,\\
			&E_{8,p}=\left\lbrace (x,y)\in A_\Omega:u_p(x)\ge k+1 \text{ and } k\le u_p(y)< k+1\right\rbrace,\\
			&E_{9,p}=\left\lbrace (x,y)\in A_\Omega:u_p(x)\ge k+1 \text{ and } u_p(y)\ge k+1\right\rbrace.\\
		\end{align*}
		
		Before discussing the limiting cases, we show that the function $Z_k$ could be got by the sequence
		\begin{align*}
			\left\lbrace 
			\frac{|u_p(x)-u_p(y)|^{p-2}\left( u_p(x)-u_p(y)\right)   }{|x-y|^{\left( \frac{N}{p}+s\right)(p-1) }}\mathcal{X}_{\left\lbrace (x,y)\in A_\Omega:u_p(x)<k\text{ and }u_p(y)<k\right\rbrace }
			\right\rbrace .
		\end{align*}
		By limiting $p<\frac{N}{N+s-1}$, we have
		\begin{align*}
			&\quad\int_{E_{1,p}}\left( 
			\frac{|u_p(x)-u_p(y)|^{p-2}\left( u_p(x)-u_p(y)\right)   }{|x-y|^{\left( \frac{N}{p}+s\right)(p-1) }}
			\right)^r dxdy\\
			&\le \text{meas}\left( \Omega_0\right)^{2\left( 1-\frac{(p-1)}{p}r\right) }\left( k\|f\|_{L^1(\Omega)}\right)^{\frac{p-1}{p}r} \\
			&\le \text{max}\left\lbrace \text{meas}\left( \Omega_0\right)^2,1\right\rbrace k^{\frac{p-1}{p}r} \|f\|_{L^1(\Omega)}^{\frac{p-1}{p}r}\\
			&\le C(\Omega,k)\|f\|_{L^1(\Omega)}^{\frac{p-1}{p}r}
		\end{align*}
		for fixed $p$ and any $r<\frac{p}{p-1}$. As $p$ goes to $1^+$, up to a subsequence, there holds
		\begin{align}
			\label{4.8}
			\frac{|u_p(x)-u_p(y)|^{p-2}\left( u_p(x)-u_p(y)\right)   }{|x-y|^{\left( \frac{N}{p}+s\right)(p-1) }}\mathcal{X}_{\left\lbrace (x,y)\in A_\Omega:u_p(x)<k\text{ and }u_p(y)<k\right\rbrace }\rightharpoonup Z_k
		\end{align}
		in $L^q(A_\Omega), q>\frac{N}{1-s}$. Moreover, for any $\rho\in \left[ 1,\frac{N}{N+s-1}\right)$, since
		\begin{align*}
			\frac{\varphi(x)-\varphi(y)}{|x-y|^{\frac{N}{p}+s}}\rightarrow\frac{\varphi(x)-\varphi(y)}{|x-y|^{N+s}}\quad\text{a.e. in }A_\Omega
		\end{align*}
		and
		\begin{align*}
			\left( \int_{A_\Omega}\left( \frac{\varphi(x)-\varphi(y)}{|x-y|^{\frac{N}{p}+s}}\right)^\rho dxdy \right)^\frac{1}{\rho}&\le \left[ \text{meas}\left(  \Omega\right)  \int_{0}^{l+1}\frac{C(\varphi)r^{N-1}}{r^{\left( \frac{N}{p}+s-1\right)\rho }}dr\right]^\frac{1}{\rho}\\
			&\le \left[ C(\varphi,\Omega,N)\frac{r^{N-\rho\left( \frac{N}{p}+s-1\right) }}{N-\rho\left( \frac{N}{p}+s-1\right) }\Bigg|_0^{l+1}\right]^\frac{1}{\rho}\\
			&\le C(\varphi,\Omega,N,s,\rho)(l+1)^{N+\rho}\\
			&\le C(\varphi,\Omega,N,s,\rho),
		\end{align*}
		where $l$ is the diameter of $\Omega$. It is easy to prove that
		\begin{align}
			\label{4.9}
			\frac{\varphi(x)-\varphi(y)}{|x-y|^{\frac{N}{p}+s}}\rightarrow\frac{\varphi(x)-\varphi(y)}{|x-y|^{N+s}}\quad\text{in }L^\rho(A_\Omega).
		\end{align}\par 
		(1) In $E_{1,p}$, according to \eqref{4.8}, \eqref{4.9} and $S(u_p)\to S(u)$ a.e.  in $\mathbb{R}^N$, we have
		\begin{align*}
			&\quad \lim\limits_{p\rightarrow1}\int_{E_{1,p}}\frac{|u_p(x)-u_p(y)|^{p-2}\left( u_p(x)-u_p(y)\right)\left( \varphi(x)-\varphi(y) \right)   }{|x-y|^{N+sp}}\frac{S(u_p)(x)+S(u_p)(y)}{2}dxdy\\
			&=\int_{A_\Omega}Z_k\frac{\varphi(x)-\varphi(y)}{|x-y|^{N+s}}\frac{S(u)(x)+S(u)(y)}{2}dxdy.
		\end{align*}
		
		(2) In $E_{2,p}$, by (2) in Definition \ref{def2.1},
		\begin{align}
			\label{4.10}
			\begin{aligned}
				\lim\limits_{k\rightarrow+\infty}\lim\limits_{p\rightarrow1}\int_{E_{2,p}}&\frac{|u_p(x)-u_p(y)|^{p-2}\left( u_p(x)-u_p(y)\right)\left( \varphi(x)-\varphi(y) \right)   }{|x-y|^{N+sp}}\\
				&\times\frac{S(u_p)(x)+S(u_p)(y)}{2}dxdy=0.
			\end{aligned}
		\end{align}
		By replacing $E_{2,p}$ with $E_{7,p}$, we have
		\begin{align*}
			\begin{aligned}
				\lim\limits_{k\rightarrow+\infty}\lim\limits_{p\rightarrow1}\int_{E_{7,p}}&\frac{|u_p(x)-u_p(y)|^{p-2}\left( u_p(x)-u_p(y)\right)\left( \varphi(x)-\varphi(y) \right)   }{|x-y|^{N+sp}}\\
				&\times\frac{S(u_p)(x)+S(u_p)(y)}{2}dxdy=0.
			\end{aligned}
		\end{align*}
		
		(3) Similar to (1), note that
		\begin{align*}
			E_{1,p}\cup E_{3,p}\cup E_{4,p}\cup E_{5,p}=\left\lbrace (x,y)\in A_\Omega:u_p(x)<k+1 \text{ and } u_p(y)<k+1\right\rbrace, 
		\end{align*}
		we have
		\begin{align*}
			\lim\limits_{p\rightarrow1}&\quad\int_{E_{1,p}\cup E_{3,p}\cup E_{4,p}\cup E_{5,p}}\frac{|u_p(x)-u_p(y)|^{p-2}\left( u_p(x)-u_p(y)\right)\left( \varphi(x)-\varphi(y) \right)   }{|x-y|^{N+sp}}\\
			&\qquad\qquad\qquad\qquad\qquad\times\frac{S(u_p)(x)+S(u_p)(y)}{2}dxdy\\
			&=\int_{A_\Omega}Z_{k+1}\frac{\varphi(x)-\varphi(y)}{|x-y|^{N+s}}\frac{S(u)(x)+S(u)(y)}{2}dxdy.
		\end{align*}
		
		(4) In $E_{6,p},E_{8,p}$ and $E_{9,p}$, because of $S(u_p)(x)=S(u_p)(y)=0$, we have
		\begin{align*}
			\int_{E_{6,p}\cup E_{8,p}\cup E_{9,p}}&\frac{|u_p(x)-u_p(y)|^{p-2}\left( u_p(x)-u_p(y)\right)\left( \varphi(x)-\varphi(y) \right)   }{|x-y|^{N+sp}}\\
			&\times\frac{S(u_p)(x)+S(u_p)(y)}{2}dxdy=0.
		\end{align*}
		
		To sum up,
		\begin{align}
			\label{4.11}
				&\quad\lim\limits_{p\rightarrow1}\int_{A_\Omega}\frac{|u_p(x)-u_p(y)|^{p-2}\left( u_p(x)-u_p(y)\right)\left( \varphi(x)-\varphi(y) \right)   }{|x-y|^{N+sp}}\frac{S(u_p)(x)+S(u_p)(y)}{2}dxdy\notag\\
				&=\lim\limits_{k\rightarrow+\infty}\lim\limits_{p\rightarrow1}\int_{A_\Omega}\frac{|u_p(x)-u_p(y)|^{p-2}\left( u_p(x)-u_p(y)\right)\left( \varphi(x)-\varphi(y) \right)   }{|x-y|^{N+sp}}\notag\\
				&\qquad\qquad\qquad\quad\quad\times\frac{S(u_p)(x)+S(u_p)(y)}{2}dxdy\notag\\
				&=\lim\limits_{k\rightarrow+\infty}\int_{A_\Omega}Z_{k+1}\frac{\varphi(x)-\varphi(y)}{|x-y|^{N+s}}\frac{S(u)(x)+S(u)(y)}{2}dxdy\notag\\
				&=\int_{A_\Omega}Z_{\infty}\frac{\varphi(x)-\varphi(y)}{|x-y|^{N+s}}\frac{S(u)(x)+S(u)(y)}{2}dxdy,
		\end{align}
		where we used \eqref{4.4} in the last line.
		
		As to $I_2$, we still study it in the domain $E_{i,p}(i\in \left\lbrace 1,2,\dots,9\right\rbrace )$, similar to (2) and (4) in the discussion of $I_1$, we have
		\begin{align}
			\label{4.12}
			\begin{aligned}
				\lim\limits_{k\rightarrow+\infty}\lim\limits_{p\rightarrow1}\int_{E_{2,p}\cup E_{7,p}}&\frac{|u_p(x)-u_p(y)|^{p-2}\left( u_p(x)-u_p(y)\right)\left( S(u_p)(x)-S(u_p)(y) \right)   }{|x-y|^{N+sp}}\\
				&\times\frac{\varphi(x)+\varphi(y)}{2}dxdy=0
			\end{aligned}
		\end{align}
		and
		\begin{align}
			\label{4.13}
			\begin{aligned}
				\lim\limits_{p\rightarrow1}\int_{E_{6,p}\cup E_{8,p}\cup E_{9,p}}&\frac{|u_p(x)-u_p(y)|^{p-2}\left( u_p(x)-u_p(y)\right)\left( S(u_p)(x)-S(u_p)(y) \right)   }{|x-y|^{N+sp}}\\
				&\times\frac{\varphi(x)+\varphi(y)}{2}dxdy=0.
			\end{aligned}
		\end{align}
		Finally, we deal with
		\begin{align*}
			\int_{E_{1,p}\cup E_{3,p}\cup E_{4,p}\cup E_{5,p}}&\frac{|u_p(x)-u_p(y)|^{p-2}\left( u_p(x)-u_p(y)\right)\left( S(u_p)(x)-S(u_p)(y) \right)   }{|x-y|^{N+sp}}\\
			&\times\frac{\varphi(x)+\varphi(y)}{2}dxdy.
		\end{align*}
		By mean-value theorem, observe that
		\begin{align*}
			S(u_p)(x)-S(u_p)(y)
		\end{align*}
		could be written as
		\begin{align*}
			S'\left( \theta(u_p(x),u_p(y))\right)(u_p(x)-u_p(y))
		\end{align*}
		with $S'\left( \theta(u_p(x),u_p(y))\right)\rightarrow S'\left( \theta(u(x),u(y))\right) \text{ a.e. in }\mathbb{R}^N\times\mathbb{R}^N$. By hypothesis of $S$ and $\varphi$, we have
		\begin{align*}
			&\quad\frac{|u_p(x)-u_p(y)|^{p-2}\left( u_p(x)-u_p(y)\right)\left( S(u_p)(x)-S(u_p)(y) \right)   }{|x-y|^{N+sp}}\frac{\varphi(x)+\varphi(y)}{2}\\
			&=\frac{|u_p(x)-u_p(y)|^{p}}{|x-y|^{N+sp}}S'\left( \theta(u_p(x),u_p(y))\right)\frac{\varphi(x)+\varphi(y)}{2} \ge 0.
		\end{align*}
		Therefore, using Fatou's lemma, we could verify that,
		\begin{align}
			\label{4.14}
				&\quad\int_{\left\lbrace (x,y)\in A_\Omega:u(x)<k+1 \text{ and } u(y)<k+1\right\rbrace}\frac{|u(x)-u(y)|}{|x-y|^{N+s}}S'\left( \theta(u(x),u(y))\right)\frac{\varphi(x)+\varphi(y)}{2}dxdy\notag\\
				&=\int_{A_\Omega}\lim\limits_{p\rightarrow1}\frac{|u_p(x)-u_p(y)|^p}{|x-y|^{N+sp}}\mathcal{X}_{\left\lbrace (x,y)\in A_\Omega:u_p(x)<k+1 \text{ and } u_p(y)<k+1\right\rbrace}S'\left( \theta(u_p(x),u_p(y))\right)\notag\\
				&\qquad\quad\times\frac{\varphi(x)+\varphi(y)}{2}dxdy\notag\\
				&\le \lim\limits_{p\rightarrow1}\int_{\left\lbrace (x,y)\in A_\Omega:u_p(x)<k+1 \text{ and } u_p(y)<k+1\right\rbrace}\frac{|u_p(x)-u_p(y)|^p}{|x-y|^{N+sp}}S'\left( \theta(u_p(x),u_p(y))\right)\notag\\
				&\qquad\quad\times\frac{\varphi(x)+\varphi(y)}{2}dxdy.
		\end{align}
		By Step 3, in the domain $\left\lbrace (x,y)\in A_\Omega:u(x)<k+1 \text{ and } u(y)<k+1\right\rbrace$,
		\begin{align*}
			|u(x)-u(y)|=Z_{k+1}(u(x)-u(y)),
		\end{align*}
		we have
		\begin{align}
			\label{4.15}
				&\quad\int_{\left\lbrace (x,y)\in A_\Omega:u(x)<k+1 \text{ and } u(y)<k+1\right\rbrace}\frac{|u(x)-u(y)|}{|x-y|^{N+s}}S'\left( \theta(u(x),u(y))\right)\frac{\varphi(x)+\varphi(y)}{2}dxdy\notag\\
				&=\int_{A_\Omega}Z_{k+1}\frac{u(x)-u(y)}{|x-y|^{N+s}}S'\left( \theta(u(x),u(y))\right)\frac{\varphi(x)+\varphi(y)}{2}dxdy\notag\\
				&=\int_{A_\Omega}Z_{k+1}\frac{S(u)(x)-S(u)(y)}{|x-y|^{N+s}}\frac{\varphi(x)+\varphi(y)}{2}dxdy.
		\end{align}
		Combining \eqref{4.12}--\eqref{4.15}, we conclude that
		\begin{align}
			\label{4.16}
				&\quad\int_{A_\Omega}Z_{\infty}\frac{S(u)(x)-S(u)(y)}{|x-y|^{N+s}}\frac{\varphi(x)+\varphi(y)}{2}dxdy\notag\\
				&=\lim\limits_{k\rightarrow+\infty}\int_{A_\Omega}Z_{k+1}\frac{S(u)(x)-S(u)(y)}{|x-y|^{N+s}}\frac{\varphi(x)+\varphi(y)}{2}dxdy\notag\\
				&\le \lim\limits_{k\rightarrow+\infty}\lim\limits_{p\rightarrow1}\int_{A_\Omega}\frac{|u_p(x)-u_p(y)|^{p}}{|x-y|^{N+sp}}S'\left( \theta(u_p(x),u_p(y))\right)\frac{\varphi(x)+\varphi(y)}{2}dxdy\notag\\
				&\le \lim\limits_{k\rightarrow+\infty}\lim\limits_{p\rightarrow1}\int_{A_\Omega}\frac{|u_p(x)-u_p(y)|^{p-2}\left( u_p(x)-u_p(y)\right)\left( S(u_p)(x)-S(u_p)(y) \right)   }{|x-y|^{N+sp}}\notag\\
				&\qquad\qquad\qquad\qquad\times\frac{\varphi(x)+\varphi(y)}{2}dxdy.
		\end{align}
		
		Finally, we consider
		\begin{align*}
			\int_{\Omega\times\left( \mathbb{R}^N\backslash\Omega_0\right) }\frac{|u_p(x)-u_p(y)|^{p-2}\left( u_p(x)-u_p(y)\right)\left[ \left( S(u_p)\varphi\right)(x)-\left( S(u_p)\varphi\right)(y) \right]  }{|x-y|^{N+sp}}dxdy,
		\end{align*}
		which could be written as
		\begin{align*}
			\int_{\Omega\times\left( \mathbb{R}^N\backslash\Omega_0\right) }\frac{|u_p(x)|^{p-2}u_p(x)\left( S(u_p)\varphi\right)(x) }{|x-y|^{N+sp}}dxdy.
		\end{align*}
		As $p$ goes to $1^+$, we have
		\begin{align}
			\label{4.17}
			\frac{1}{|x-y|^{N+sp}}\rightarrow\frac{1}{|x-y|^{N+s}}\quad \text{in } L^\rho\left( \mathbb{R}^N\backslash\Omega_0\right)
		\end{align}
		for any $\rho\in\left[ 1,+\infty\right)$ and $x\in\Omega$. Therefore,
		\begin{align*}
			&\quad\ \int_{\Omega\times\left( \mathbb{R}^N\backslash\Omega_0\right) }\frac{|u_p(x)|^{p-2}\left( u_p(x)\right)\left( S(u_p)\varphi\right)(x) }{|x-y|^{N+sp}}dxdy\\
			&=\int_{\Omega\times\left( \mathbb{R}^N\backslash\Omega_0\right) }\frac{|T_k(u_p)(x)|^{p-2}T_k(u_p)(x)\left( S(u_p)\varphi\right)(x)  }{|x-y|^{N+sp}}dxdy\\
			&=\int_{\Omega\times\left( \mathbb{R}^N\backslash\Omega_0\right) }|T_k(u_p)(x)|^{p-2}T_k(u_p)(x)\left( S(u_p)\varphi\right)(x)\\
			&\qquad\qquad\qquad\ \times\left( \frac{1}{|x-y|^{N+sp}}-\frac{1}{|x-y|^{N+s}}+\frac{1}{|x-y|^{N+s}}\right)dxdy\\
			&=I_3+I_4,
		\end{align*}
		where
		\begin{align*}
			&I_3=\int_{\Omega\times\left( \mathbb{R}^N\backslash\Omega_0\right) }|T_k(u_p)(x)|^{p-2}T_k(u_p)(x)\left( S(u_p)\varphi\right)(x)\left( \frac{1}{|x-y|^{N+sp}}-\frac{1}{|x-y|^{N+s}}\right)dxdy,\\
			&I_4=\int_{\Omega\times\left( \mathbb{R}^N\backslash\Omega_0\right) }|T_k(u_p)(x)|^{p-2}T_k(u_p)(x)\left( S(u_p)\varphi\right)(x)\frac{1}{|x-y|^{N+s}}dxdy.
		\end{align*}
		By \eqref{4.17}, we have
		\begin{align*}
			\lim\limits_{p\rightarrow1}I_3=0
		\end{align*}
		and
		\begin{align*}
			\lim\limits_{p\rightarrow1}I_4&=\lim\limits_{p\rightarrow1}\int_{\Omega\times\left( \mathbb{R}^N\backslash\Omega_0\right) }\frac{|u_p(x)|^{p-2}u_p(x)\left( S(u_p)\varphi\right)(x) }{|x-y|^{N+s}}dxdy\\
			&=\int_{\Omega\times\left( \mathbb{R}^N\backslash\Omega_0\right) }Z\frac{\left( S(u)\varphi\right)(x)}{|x-y|^{N+s}}dxdy\\
			&=\int_{\Omega\times\left( \mathbb{R}^N\backslash\Omega_0\right) }Z_k\frac{\left( S(u)\varphi\right)(x)}{|x-y|^{N+s}}dxdy.
		\end{align*}
		As a consequence,
		\begin{align}
			\label{4.18}
				&\quad\ \lim\limits_{p\rightarrow1}\int_{\Omega\times\left( \mathbb{R}^N\backslash\Omega_0\right) }\frac{|u_p(x)-u_p(y)|^{p-2}\left( u_p(x)-u_p(y)\right)\left[ \left( S(u_p)\varphi\right)(x)-\left( S(u_p)\varphi\right)(y) \right]  }{|x-y|^{N+sp}}dxdy\notag\\
				&=\lim\limits_{k\rightarrow+\infty}\lim\limits_{p\rightarrow1}\int_{\Omega\times\left( \mathbb{R}^N\backslash\Omega_0\right) }\frac{|u_p(x)-u_p(y)|^{p-2}\left( u_p(x)-u_p(y)\right)  }{|x-y|^{N+sp}}\notag\\
				&\qquad\qquad\qquad\qquad\qquad\quad\times\left[ \left( S(u_p)\varphi\right)(x)-\left( S(u_p)\varphi\right)(y) \right]dxdy\notag\\
				&=\lim\limits_{k\rightarrow+\infty}\int_{\Omega\times\left( \mathbb{R}^N\backslash\Omega_0\right) }Z_k\frac{\left( S(u)\varphi\right)(x) }{|x-y|^{N+s}}dxdy\notag\\
				&=\int_{\Omega\times\left( \mathbb{R}^N\backslash\Omega_0\right) }Z_\infty\frac{\left( S(u)\varphi\right)(x)-\left( S(u)\varphi\right)(y) }{|x-y|^{N+s}}dxdy.
		\end{align}
		Combining \eqref{4.7}, \eqref{4.11}, \eqref{4.16} and \eqref{4.18}, and using the fact that $Z$ is anti-symmetric, we deduce
		\begin{align*}
			\int_{\mathcal{C}_\Omega}Z_{\infty}\frac{\left( S(u)\varphi\right)(x)-\left(S(u)\varphi\right)(y) }{|x-y|^{N+s}}dxdy\le \int_{\Omega}fS(u)\varphi dx.
		\end{align*}
	This finished the proof. $\quad \Box$
	
	\medskip
	
\subsection*{Data availability}
No data was used for the research described in the article.
	
\subsection*{Acknowledgments}
This work was supported by the National Natural Science Foundation of China (No. 12071098) and the Fundamental Research Funds for the Central Universities (No. 2022FRFK060022).


\begin{thebibliography}{[a]}
		
		\bibitem{AAB19} B. Abdellaoui, A. Attar and R. Bentifour, On the fractional $p$-Laplacian equations with weight and general datum, Adv. Nonlinear Anal. 8 (2019) 144--174.
		
		\bibitem{AAB10} N. Alibaud, B. Andreianov and M. Bendahmane, Renormalized solutions of the fractional Laplace equations, C. R. Acad. Sci. Paris. Ser. I 348 (2010) 759--762.
		
		\bibitem{ABCM01} F. Andreu, C. Ballester, V. Caselles and J.M. Maz\'{o}n, The Dirichlet problem for the total variation flow, J. Funct. Anal. 180 (2) (2001) 347--403.
		
		\bibitem{ABCM00} F. Andreu, C. Ballester, V. Caselles and J.M. Maz\'{o}n, Minimizing total variation flow, C. R. Acad. Sci. Paria Sr. I Math. 331 (11) (2000) 867--872.
		
		\bibitem{ABCM012} F. Andreu, C. Ballester, V. Caselles and J.M. Maz\'{o}n, Minimizing total variation flow, Differ. Integr. Equ. 14 (3) (2001) 321--360.
		
		\bibitem{ACDM02} F. Andreu, V. Caselles, J.I. D\'{i}az and J.M. Maz\'{o}n, Some qualitative properties for the total variation flow, J. Funct. Anal. 188 (2) (2002) 516--547.
		
		\bibitem{ACM04} F. Andreu, V. Caselles and J.M. Maz\'{o}n, Parabolic quasilinear equations minimizing linear growth functionals: $L^1$-theory, Progress in Mathematics 233 (2004) 213--269.
		
		\bibitem{AMRT08} F. Andreu, J.M. Maz\'{o}n, J.D. Rossi and J. Toledo, A nonlocal $p$-Laplacian evolution equation with Neumann boundary conditions, J. Math. Pures Appl. 90 (9) (2008) 201--227.
		
		\bibitem{AMRT09} F. Andreu, J.M. Maz\'{o}n, J.D. Rossi and J. Toledo, A nonlocal $p$-Laplacian evolution equation with nonhomogeneous Dirichlet boundary conditions, SIAM J. Math. Anal. 40 (5) (2008) 1815--1851.
		
		\bibitem{An83} G. Anzellotti, Pairings between measures and bounded functions and compensated compactness, Ann. Math. Pura Appl. 135 (4) (1983) 293--318.
		
		\bibitem{BR18} A. Bahrouni and V. R\u{a}dulescu, On a new fractional Sobolev space and applications to nonlocal variational problems with variable exponent, Discrete Contin. Dyn. Syst. Ser. S 11 (3) (2018) 379--389.
		
		\bibitem {BM97} D. Blanchard and F. Murat, Renormalised solutions of nonlinear parabolic problems with $L^1$ data: Existence and uniqueness, Proc. Roy. Soc. Edinburgh Sect. A 127 (6) (1997) 1137--1152.
		
		\bibitem {BM01} D. Blanchard, F. Murat and H. Redwane, Existence and uniqueness of a renormalized solution for a fairly general class of nonlinear parabolic problems, J. Differential Equations 177 (2) (2001) 331--374.
		
		\bibitem {BR98} D. Blanchard and H. Redwane, Renormalized solutions for a class of nonlinear evolution problems, J. Math. Pure Appl. 77 (1998) 117--151.
		
		\bibitem{BG89} L. Boccardo and T. Gallou\"{e}t, Nonlinear elliptic and parabolic equations involving measure data, J. Funct. Anal. 87 (1989) 149--169.
		
		\bibitem{BG92} L. Boccardo and T. Gallou\"{e}t, Nonlinear elliptic equations with right hand side measures, Comm. Partial Differential Equations 17 (1992) 641--655.
		
		\bibitem{BGD93} L. Boccardo, D. Giachetti, J.I. Diaz and F. Murat, Existence and regularity of renormalized solutions for some elliptic problems involving derivations of nonlinear terms, J. Differential Equations 106  (1993) 215--237.
		
		\bibitem{BU23} C. Bucur, Solutions of the fractional 1-Laplacian: existence, asymptotics and flatness results, arXiv:2310.13656.
		
		\bibitem{BU21} C. Bucur, S. Dipierro, L. Lombardini, J.M. Maz\'{o}n and E. Valdinoci, $(s,p)$-harmonic approximation of functions of least $W^{s,1}$-seminorm, Int. Math. Res. Notices 2023 (2) (2021) 1173--1235.
		
		\bibitem{CF99} G. Chen and H. Frid, Divergence-measure fields and hyperbolic conservation laws, Arch. Ration. Mech. Anal. 147 (1999) 89--118.
		
		\bibitem{CT03} M. Cicalese and C. Trombetti, Asymptotic behaviour of solutions to $p$‐Laplacian equation, Asymptot. Anal. 35 (2003) 27--40.
		
		\bibitem{C17} M. Cozzi, Regularity results and Harnack inequalities for minimizers and solutions of nonlocal problems: A unified approach via fractional De Giorgi classes, J. Funct. Anal. 272 (2017) 4762--4837.
		
		\bibitem{DMOP99} G. Dal Maso, F. Murat, L. Orsina and A. Prignet, Renormalized solutions of elliptic equations with general measure data, Ann. Sc. Norm. Super. Pisa CI. Sci. 28 (4) (1999) 741--808.
		
		\bibitem{De99} F. Demengel, On some nonlinear partial differential equations involving the ``1"-Laplacian and critical Sobolev exponent, ESAIM Control Optim. Calc. Var. 4 (1999) 667--686.
		
		\bibitem{DPV12} E. Di Nezza, G. Palatucci and E. Valdinoci, Hitchhiker's guide to the fractional Sobolev spaces, Bull. Sci. Math. 136 (2012) 521--573.
		
		\bibitem{DL89} R.J. Diperna and P.L. Lions, On the cauchy problem for Boltzmann equations: Global existence and weak stability, Ann. of Math. 130 (1989) 321--366.
		
		\bibitem{KRB17} U. Kaufmann, J.D. Rossi and R. Bidel, Fractional Sobolev space with variable exponents and fractional $p(x)$-Laplacians, Electron. J. Qual. Theory Differ. Equ. 10 (2017) Paper No.76.
		
		\bibitem{Ka90} B. Kawohl, On a familiy of torsional creep problems, J. Reine Angew. Math. 410 (1990) 1--22.
		
		\bibitem{KPU11} K.H. Kenneth, F. Petitts and S. Ulusoy, A duality approach to the fractional Laplacian with measure data, Publ. Math. 55 (1) (2011) 151--161.
		
		\bibitem{KMS15} T. Kuusi, G. Mingione and Y. Sire, Nonlocal equations with measure data, Comm. Math. Phys. 377 (2015) 1317--1368.
		
		\bibitem{LS22} M. Latorre and S. Segura de Le\'{o}n, Existence and uniqueness for the inhomogeneous 1-Laplace evolution equation revisited, Rev. Real Acad. Cienc. Exactas Fis. Nat. Ser. A-Mat. (2022) 116--185.
		
		\bibitem{LPPS15} T. Leonori, I. Peral, A. Primo and F. Soria, Basic esimates for solutions of a class of nonlocal elliptic and parabolic equation, Discrete contin. Dyn. syst. 35 (12) (2015) 6031--6068.
		
		\bibitem{MY23} A. Matsoukas and N. Yannakakis, The double phase Dirichlet problem when the lowest exponent is equal to 1, J. Math. Anal. Appl. 526 (2023) 127270.
		
		\bibitem{MRT19}  J.M. Maz\'{o}n, J.D. Rossi and J. Toledo, Nonlocal perimeter, curvature and minimal surfaces for measurable sets, J. Anal. Math. 138 (1) (2019) 235--279.
		
		\bibitem{MST08} A. Mercaldo, S. Segura de Le\'{o}n and C. Trombetti, On the behaviour of the solutions to $p$-Laplacian equations as $p$ goes to 1, Publ. Mat. 52 (2008) 377--411.
		
		\bibitem{MST09} A. Mercaldo, S. Segura de Le\'{o}n and C. Trombetti, On the solutions to 1-Laplacian equation with $L^1$ data, J. Funct. Anal. 256 (8) (2009) 2387--2416.
		
		\bibitem{MF23} M. Novaga and F. Onoue, Local H\"{o}lder regularity of minimizers for nonlocal variational problems, Commun. Contemp. Math. 25 (10) (2023) 2250058.
		
		\bibitem{SW15} S. Segura de Le\'{o}n and C.M. Webler, Global existence and uniqueness for the inhomogeneous 1-Laplace evolution equation, Nonlinear Differ. Equ. Appl. 22 (2015) 1213--1246.
		
		\bibitem{ZZ20} C. Zhang and X. Zhang, Renormalized solutions for the fractional $p(x)$-Laplacian equation with $L^1$ data, Nonlinear Analysis 190 (2020) 111610.
		
		\bibitem{ZS10} C. Zhang and S. Zhou, Renormalized and entropy solutions for nonlinear parabolic equations with variable exponents and $L^1$ data, J. Differential Equations 248 (2010) 1376--1400.
		
	\end{thebibliography}
\end{document}